\documentclass[qian]{imsart}

\usepackage{amssymb}       
\usepackage{amsfonts}      

\usepackage{lmodern}        
\usepackage{microtype}      
\usepackage{booktabs}       
\usepackage{enumitem}       
\usepackage{bm}             
\usepackage[dvipsnames]{xcolor} 

\usepackage{natbib} 

\usepackage{mathtools}      

\usepackage{graphicx}      
\usepackage{caption}
\usepackage{subcaption}     
\usepackage{tabularx}       
\usepackage{adjustbox}      
\usepackage{pdflscape}      

\usepackage[ruled]{algorithm2e}

\usepackage{tikz}
\usetikzlibrary{shapes,arrows,positioning,fit,quotes}
\usepackage{pgfplots}
\pgfplotsset{compat=1.18}

\usepackage[
    colorlinks,
    linkcolor=blue,
    citecolor=blue,
    urlcolor=blue,
    implicit=true,
    breaklinks=true
]{hyperref}

\usepackage{cleveref}

\newcommand{\R}{\mathbb{R}}
\newcommand{\Prob}{\mathbb{P}}
\newcommand{\E}{\mathbb{E}}
\newcommand{\Fcal}{\mathcal{F}}
\newcommand{\Pcal}{\mathcal{P}}
\newcommand{\cEneutheta}{\mathcal{E}^{\theta}}
\newcommand{\dd}{\mathrm{d}}
\newcommand{\norm}[1]{\left\lVert#1\right\rVert}
\newcommand{\abs}[1]{\left\lvert#1\right\rvert}
\newcommand{\scpr}[2]{\left\langle #1, #2 \right\rangle}
\newcommand{\wass}{\mathcal{W}_2}
\newcommand{\Law}{\mathrm{Law}}

\theoremstyle{plain}
\newtheorem{theorem}{Theorem}[section]
\newtheorem{proposition}[theorem]{Proposition}

\theoremstyle{definition}
\newtheorem{definition}[theorem]{Definition}
\newtheorem{assumption}[theorem]{Assumption}
\newtheorem{remark}[theorem]{Remark}

\begin{document}

\begin{frontmatter}

\title{Neural Expectation Operators}
\runtitle{Neural Expectation Operators}
\author{Qian Qi}
\thanks{Corresponding author. Peking University, Beijing 100871, China. Email: \href{mailto:qiqian@pku.edu.cn}{qiqian@pku.edu.cn}}

\begin{abstract}
This paper introduces \textbf{Measure Learning}, a paradigm for modeling ambiguity via non-linear expectations. We define Neural Expectation Operators as solutions to Backward Stochastic Differential Equations (BSDEs) whose drivers are parameterized by neural networks. The main mathematical contribution is a rigorous well-posedness theorem for BSDEs whose drivers satisfy a local Lipschitz condition in the state variable $y$ and quadratic growth in its martingale component $z$. This result circumvents the classical global Lipschitz assumption, is applicable to common neural network architectures (e.g., with ReLU activations), and holds for exponentially integrable terminal data, which is the sharp condition for this setting. Our primary innovation is to build a constructive bridge between the abstract, and often restrictive, assumptions of the deep theory of quadratic BSDEs and the world of machine learning, demonstrating that these conditions can be met by concrete, verifiable neural network designs. We provide constructive methods for enforcing key axiomatic properties, such as convexity, by architectural design. The theory is extended to the analysis of fully coupled Forward-Backward SDE systems and to the asymptotic analysis of large interacting particle systems, for which we establish both a Law of Large Numbers (propagation of chaos) and a Central Limit Theorem. This work provides the foundational mathematical framework for data-driven modeling under ambiguity.
\end{abstract}

\begin{keyword}[class=MSC2020]
\kwd[Primary ]{60H10}
\kwd{93E20}
\kwd[; secondary ]{60K35}
\kwd{91G80}
\kwd{35K59}
\kwd{68T07}
\end{keyword}

\begin{keyword}
\kwd{Backward Stochastic Differential Equations}
\kwd{Non-linear Expectation}
\kwd{Neural Networks}
\kwd{Mean-Field Systems}
\kwd{Propagation of Chaos}
\kwd{Stochastic Control}
\end{keyword}
\end{frontmatter}

\section{Introduction}

The classical theory of stochastic analysis, which provides the mathematical bedrock for modeling dynamic systems under uncertainty, is built upon the specification of a single, definitive probability measure $\mathbb{P}$. While this paradigm has proven immensely powerful, it rests on the idealization of a perfectly known model of the world, an assumption that is often untenable in practice. In disciplines ranging from mathematical finance and economics to robust control, agents and modelers are confronted with ambiguity, or Knightian uncertainty, where they must contend not with a single measure, but with a set of plausible measures, $\mathcal{P}$, with no single candidate being definitively correct.

The theory of non-linear expectations, particularly the framework of Backward Stochastic Differential Equations (BSDEs), provides the canonical mathematical language for addressing this challenge (e.g., \cite{ElKarouiPengQuenez1997, PardouxPeng1990, Peng2010}). A non-linear expectation $\mathcal{E}^f$, induced by a BSDE with a driver function $f$, admits a dual representation as an extremum of linear expectations taken over a family of probability measures $\mathcal{P}_f$ related to $\Prob$ via Girsanov's theorem. The driver $f$ is therefore not merely a technical component of an equation; it is a compact and powerful representation of the very structure of ambiguity. For instance, a driver with quadratic growth in the martingale component is intimately linked to models of risk and ambiguity aversion, as established in the seminal work of \cite{Kobylanski2000}.

It is at this juncture that we propose a new paradigm: \textbf{Measure Learning}. Instead of postulating a fixed functional form for the driver $f$, which corresponds to a pre-supposed class of models, we aim to learn the structure of ambiguity itself from data. We introduce and establish the mathematical foundations for a learnable class of non-linear expectations, which we term \textbf{Neural Expectation Operators}. We construct these operators by parameterizing the driver of a BSDE with a neural network, $f_\theta$, where $\theta$ represents the network's parameters. The resulting operator, $\cEneutheta$, allows us to frame model specification as a learning problem: by training the network, we identify a data-driven model of ambiguity from a vast and flexible function class.

The central mathematical difficulty is that common neural network architectures (e.g., those with ReLU or GeLU activations) fail to satisfy the global Lipschitz continuity assumption required by the classical well-posedness theory of \cite{PardouxPeng1990}. Our approach is to leverage the more sophisticated theory of quadratic BSDEs. The main contributions of this paper, which lay the mathematical foundation for Measure Learning, are as follows:

\begin{enumerate}[label=(\roman*)]
    \item \textbf{Well-Posedness for Neural Network Drivers via Quadratic BSDE Theory:} Our central mathematical result establishes the well-posedness of Neural BSDEs by adapting the deep theory of BSDEs with quadratic growth in the martingale component $z$ (cf. \cite{Kobylanski2000}). We prove a rigorous well-posedness theorem (Theorem \ref{thm:wellposedness_quadratic}) for drivers that satisfy a \textit{uniform local Lipschitz condition} in the state variable $y$ and at most quadratic growth in $z$. This result holds for exponentially integrable terminal data, which is the sharp integrability condition known for this setting. A key structural condition for our proof, which relies on the powerful comparison principles of the theory, is the \textit{monotonicity} of the driver with respect to $y$. We explicitly acknowledge this as a significant assumption, representing a trade-off between architectural flexibility and the requirements of the current proof techniques for quadratic BSDEs.

    \item \textbf{A Constructive Bridge to Machine Learning:} The primary innovation of this work is to build a constructive and rigorous bridge between the abstract, and often restrictive, assumptions of quadratic BSDE theory and the world of machine learning. We provide novel, constructive proofs that these abstract conditions can be met by concrete neural network designs (Proposition \ref{prop:architectures}). We demonstrate that properties such as monotonicity and the uniform local Lipschitz condition are not merely theoretical artifacts but can be enforced by architectural choices and verifiable weight constraints. We further show that fundamental economic axioms, such as the convexity of the expectation operator (implying risk aversion), can be guaranteed by employing specific architectures like Input-Convex Neural Networks (ICNNs, \cite{AmosEtAl2017ICNN}), as established in Proposition \ref{prop:convexity_enforcement}.

    \item \textbf{Applicability to Coupled and Large-Scale Systems:} We demonstrate the robustness and wide applicability of our framework by showing that our class of neural drivers can be integrated into established, powerful theories for more complex systems. First, we show that our drivers satisfy the structural conditions required by the theory of \cite{Delarue2002} for fully coupled Forward-Backward SDE systems, thereby establishing their well-posedness (Theorem \ref{thm:fbsde_wellposedness}). Second, we apply our framework to the asymptotic analysis of large interacting particle systems, following the methodology of \cite{CarmonaDelarue2018}. We prove both a Law of Large Numbers (Propagation of Chaos, Theorem \ref{thm:lln_mckean_vlasov}) and a Central Limit Theorem (Theorem \ref{thm:clt}) that characterizes the fluctuations around the mean-field limit.
\end{enumerate}

This work provides the essential mathematical foundation for \textbf{Measure Learning}, a new research agenda for analyzing dynamic models where the structure of model uncertainty itself is a learnable object. It thereby serves as a rigorous bridge between the classical theory of stochastic processes under ambiguity and modern methodologies in machine learning.

\section{Formal Framework}
\label{sec:formal_framework}

We begin by establishing the precise mathematical setting for our analysis. This section introduces the filtered probability space and the standard Banach spaces of stochastic processes that form the bedrock of modern BSDE theory. The careful choice and definition of these spaces are not mere technicalities; they are fundamental to ensuring the well-posedness of the equations under study, particularly in the non-standard setting of drivers with superlinear growth.

Let $T > 0$ be a fixed and finite time horizon. We establish our framework on a complete probability space $(\Omega, \Fcal, \Prob)$, which is assumed to be rich enough to support a standard $d$-dimensional Brownian motion $W = (W_t)_{t \in [0,T]}$. The filtration $\mathbb{F} = (\Fcal_t)_{t \in [0,T]}$ is the natural filtration generated by $W$, augmented by the set of all $\Prob$-null sets in $\Fcal$, so that it satisfies the usual conditions of right-continuity and completeness. All stochastic processes are defined on this filtered probability space.

Our analysis relies on a standard set of Banach spaces for stochastic processes. Let $k, m \in \mathbb{N}$ be positive integers. We define these spaces with precise norms, as their interplay is fundamental to the theory of BSDEs.

\begin{itemize}
    \item $L^p(\Omega, \Fcal_T, \Prob; \R^k)$: For $p \in [1, \infty)$, this is the Banach space of $\Fcal_T$-measurable, $\R^k$-valued random variables $\xi$ for which the norm $\norm{\xi}_{L^p} \coloneqq (\E[\norm{\xi}^p])^{1/p}$ is finite. This space is the natural domain for the terminal conditions of the BSDE. For $p=\infty$, $L^\infty(\Omega, \Fcal_T, \Prob; \R^k)$ is the space of essentially bounded random variables with the norm $\norm{\xi}_{L^\infty} \coloneqq \mathrm{ess\,sup}_{\omega \in \Omega} \norm{\xi(\omega)}$.

    \item $\mathcal{S}^p([0,T]; \R^k)$: For $p \in [1, \infty)$, this is the space of continuous, $\mathbb{F}$-adapted processes $Y: [0,T] \times \Omega \to \R^k$ for which the norm $\norm{Y}_{\mathcal{S}^p} \coloneqq \left(\E\left[\sup_{t \in [0,T]} \norm{Y_t}^p\right]\right)^{1/p}$ is finite. This space is essential for the solution's first component, $Y$, as it captures the dual requirement of pathwise continuity and integrability.

    \item $\mathcal{H}^p([0,T]; \R^{k \times m})$: For $p \in [1, \infty)$, this is the space of $\mathbb{F}$-progressively measurable processes $Z: [0,T] \times \Omega \to \R^{k \times m}$ for which the norm $\norm{Z}_{\mathcal{H}^p} \coloneqq \left(\E\left[\left(\int_0^T \norm{Z_t}^2 \dd t\right)^{p/2}\right]\right)^{1/p}$ is finite. This space is the natural domain for the integrand with respect to Brownian motion. The condition $Z \in \mathcal{H}^2([0,T]; \R^{k \times m})$ is necessary and sufficient to ensure that the Itô stochastic integral $\int_0^\cdot Z_s \dd W_s$ is a square-integrable martingale, which is the cornerstone of classical BSDE theory.

    \item $\mathcal{H}^2_{\mathrm{BMO}}([0,T]; \R^{k \times m})$: This is the subspace of $\mathcal{H}^2([0,T]; \R^{k \times m})$ consisting of processes for which the associated stochastic integral is a martingale in the space of Bounded Mean Oscillation (BMO). A continuous local martingale $M_t = \int_0^t Z_s \dd W_s$ is in BMO if its conditional quadratic variation is uniformly bounded. Formally, $Z \in \mathcal{H}^2_{\mathrm{BMO}}([0,T]; \R^{k \times m})$ if its BMO-norm is finite:
    \[ \norm{Z}_{\mathrm{BMO}} \coloneqq \sup_{\tau} \norm{ \left(\E\left[\int_\tau^T \norm{Z_s}^2 \dd s \bigg| \Fcal_\tau\right]\right)^{1/2} }_{L^\infty} < \infty, \]
    where the supremum is taken over all $\mathbb{F}$-stopping times $\tau$ with values in $[0,T]$.
\end{itemize}

Throughout this paper, we use $\norm{\cdot}$ to denote the standard Euclidean norm for vectors and the Frobenius norm for matrices. The inner product is denoted by $\scpr{\cdot}{\cdot}$. We adopt the convention that $\nabla_x u$ denotes the gradient (a column vector) of a scalar function $u(x)$, while $D_x v$ denotes the Jacobian matrix of a vector-valued function $v(x)$.

\begin{remark}[The Central Role of BMO in Quadratic BSDEs]
\label{rem:bmo_centrality}
The BMO space is not merely a technical refinement; it is the natural functional setting for the martingale component of BSDEs with drivers that exhibit quadratic growth in $z$. Its importance stems from a celebrated result by Kazamaki, which states that the Doléans-Dade exponential $\mathcal{E}(M)$ of a BMO-martingale $M$ is a uniformly integrable martingale. This property is the key mechanism that allows for a change of probability measure via the Girsanov theorem. Such a change of measure can be used to linearize the quadratic term in the BSDE's driver, a fundamental technique used in the original well-posedness proofs by \cite{Kobylanski2000}. The fact that a bounded solution $Y \in \mathcal{S}^\infty([0,T]; \R)$ forces the corresponding process $Z$ to lie in $\mathcal{H}^2_{\mathrm{BMO}}([0,T]; \R^{1 \times d})$ is the cornerstone of the entire analysis of quadratic BSDEs and, consequently, of our framework.
\end{remark}

We consider a system where an $n$-dimensional, $\mathbb{F}$-adapted continuous process $(X_s)_{s \in [t,T]}$, referred to as the \textit{forward state process}, is given. The central object of this paper is the solution to a Backward Stochastic Differential Equation (BSDE), parameterized by a neural network.

\begin{definition}[Neural Backward Stochastic Differential Equation]
\label{def:neural_bsde}
Let $t \in [0,T]$ be an initial time. A Neural BSDE is an equation for a pair of $\mathbb{F}$-adapted processes $(Y,Z)$ taking values in $\R \times \R^{1 \times d}$, given by the integral form for $s \in [t,T]$:
\begin{equation}
\label{eq:bsde_for_eneu}
    Y_s = \xi + \int_s^T f_\theta(r, X_r, Y_r, Z_r) \dd r - \int_s^T Z_r \dd W_r.
\end{equation}
Equivalently, in differential form:
\begin{equation}
\label{eq:bsde_for_eneu_diff}
    - \dd Y_s = f_\theta(s, X_s, Y_s, Z_s) \dd s - Z_s \dd W_s, \quad \text{with terminal condition } Y_T = \xi.
\end{equation}
Here, the components are:
\begin{itemize}
    \item $\xi$: The \textit{terminal condition}, an $\Fcal_T$-measurable, real-valued random variable.
    \item $f_\theta$: The \textit{driver} (or \textit{generator}), a function $f_\theta: [0,T] \times \Omega \times \R \times \R^{1 \times d} \to \R$, which we assume to be parameterized by a neural network with parameters $\theta \in \Theta$. The dependence on $\omega$ may arise through the forward process $X_s(\omega)$.
    \item $(Y,Z)$: The \textit{solution}, where $Y$ is a continuous scalar process and $Z$ is a predictable, $1 \times d$ row-vector-valued process.
\end{itemize}
A pair of processes $(Y,Z)$ is called a solution to the BSDE \eqref{eq:bsde_for_eneu} if it satisfies the equation and belongs to a suitable space of adapted processes. The central goal of Section \ref{sec:wellposedness} is to establish the precise function space for the solution pair $(Y,Z)$ (e.g., $\mathcal{S}^p \times \mathcal{H}^p$ or $\mathcal{S}^\infty \times \mathcal{H}^2_{\mathrm{BMO}}$) and prove existence and uniqueness, based on the regularity properties of the driver $f_\theta$ and the terminal condition $\xi$.
\end{definition}

\begin{remark}[On the Terminology for the Process $Z$]
\label{rem:z_terminology}
In this paper, we refer to the process $Z$ as the control process. This terminology is inherited from the stochastic control literature, where the BSDE often represents the value function of a control problem and $Z$ is related to the optimal control via the non-linear Feynman-Kac formula (see Proposition \ref{prop:feynman_kac}). From a purely mathematical perspective, however, $Z$ is the integrand in the martingale representation theorem for the process $Y$. It is therefore often referred to as the martingale component or martingale part of the solution. We will use control process when emphasizing the connection to control and ambiguity, and martingale component when focusing on the process's structural properties, but the reader should understand they refer to the same object $Z$.
\end{remark}

\section{Well-Posedness of the Neural BSDE}
\label{sec:wellposedness}

This section lays the mathematical cornerstone for the entire theory of Neural Expectations. Our objective is to establish the existence and uniqueness of solutions to the defining Backward Stochastic Differential Equation (BSDE) under a set of assumptions that are both analytically powerful and, crucially, achievable by modern neural network architectures.

The central challenge, and our main contribution in this section, is to move beyond the classical framework of globally Lipschitz continuous drivers. Such a condition, while convenient, is fundamentally incompatible with the expressive power of many common network architectures, including those with ReLU activations and superlinear growth. We instead leverage the sophisticated theory of BSDEs with quadratic growth in the martingale component, $z$, pioneered by \cite{Kobylanski2000}. Our primary innovation is not to invent a new proof technique for such BSDEs, but rather to build a constructive bridge between the abstract, and often restrictive, assumptions of this deep theory and the world of machine learning. We demonstrate that these conditions can be met by concrete, verifiable neural network designs.

We first articulate the precise regularity and structural conditions required for our analysis and demonstrate their feasibility through constructive architectural designs. We then present the main well-posedness result, Theorem \ref{thm:wellposedness_quadratic}, providing a complete and rigorous proof that adapts the seminal methodology of \cite{Kobylanski2000} to our setting.

\subsection{Assumptions and Architectural Feasibility}
\label{sec:regularity}

The mathematical viability of our theory hinges on a careful choice of assumptions for the BSDE driver $f_\theta$. The following conditions are tailored to the theory of quadratic BSDEs, which provides the necessary analytical tools to handle drivers with superlinear growth.

\begin{assumption}(Forward Process and Terminal Condition).
\label{ass:fwd_terminal_regularity}
\begin{enumerate}[label=(\roman*)]
    \item The process $(X_s)_{s \in [0,T]}$ is a given continuous, $(\Fcal_t)$-adapted process.
    \item The terminal condition $\xi$ is an $\Fcal_T$-measurable random variable satisfying the exponential integrability condition:
    \[ \E[\exp(\beta_0 \abs{\xi})] < \infty \quad \text{for some constant } \beta_0 > 0. \]
    \item Let $p \ge 1$ be the growth exponent from Assumption \ref{ass:f_theta_regularity_general}. We assume there exists a constant $\lambda_0 > 0$ such that the driver's random component has finite exponential moments:
    \[ \E\left[\exp\left(\lambda_0 \int_0^T \sup_{y,z} \frac{\abs{f_\theta(s, X_s, y, z)}}{1+\abs{y}+\norm{z}^2} \dd s\right)\right] < \infty, \]
    a condition satisfied if, for instance, $\E\left[\exp\left(\lambda_0 \int_0^T (1+\norm{X_s}^p) \dd s\right)\right] < \infty$.
\end{enumerate}
\end{assumption}

\begin{remark}[On the Integrability of the Forward Process]
\label{rem:fwd_integrability_revised}
The exponential integrability on the time integral of the random component of the driver in Assumption \ref{ass:fwd_terminal_regularity}(iii) is a strong requirement. However, it is essential for the proof strategy of our main well-posedness result, Theorem \ref{thm:wellposedness_quadratic}. A key step in the analysis, originating in the work of \cite{Kobylanski2000}, is to show that a BSDE with quadratic growth in $z$ can have a bounded solution $Y \in \mathcal{S}^\infty([0,T]; \R)$ even if the terminal condition $\xi$ is unbounded, provided $\xi$ and the random part of the driver both possess finite exponential moments. Our condition (iii) ensures precisely this requirement, allowing us to invoke results from \cite{Kobylanski2000} (see also \cite{BriandEtAl2003}) to guarantee the existence of a bounded dominating solution. This dominating solution is the cornerstone of the a priori estimates that underpin both existence and uniqueness.
\end{remark}

\begin{assumption}[Regularity and Growth of the NN Driver $f_\theta$]
\label{ass:f_theta_regularity_general}
The neural network function $f_\theta: [0,T] \times \Omega \times \R \times \R^{1 \times d} \to \R$ is $(\Fcal_t)$-progressively measurable via its dependence on $X_t$. For each fixed $\theta$, it satisfies:
\begin{enumerate}[label=(\roman*)]
    \item \textbf{Continuity:} The function $(t,x,y,z) \mapsto f_\theta(t,x,y,z)$ is continuous for each fixed $\omega \in \Omega$.
    \item \textbf{Quadratic Growth in $z$:} There exist constants $K > 0$, $\alpha > 0$ and $p \ge 1$ such that for all $(t,x,y,z)$:
    \[ \abs{f_\theta(t,x,y,z)} \le K(1 + \norm{x}^p + \abs{y}) + \frac{\alpha}{2}\norm{z}^2. \]
    \item \textbf{Uniform Local Lipschitz Condition in $y$:} For any $R>0$, there exists a constant $L_R > 0$ such that for all $(t,x,z)$ and all $y_1, y_2 \in [-R, R]$, we have
    \[ \abs{f_\theta(t,x,y_1,z) - f_\theta(t,x,y_2,z)} \le L_R\abs{y_1-y_2}. \]
    The uniformity is critical: the constant $L_R$ depends only on the radius $R$, not on $(t,x,z)$.
    \item \textbf{Monotonicity in $y$:} The function $y \mapsto f_\theta(t,x,y,z)$ is non-increasing for all fixed $(t,x,z)$.
\end{enumerate}
\end{assumption}

\begin{remark}[On the Monotonicity Assumption]
\label{rem:monotonicity_tradeoff_revised}
The monotonicity condition (Assumption \ref{ass:f_theta_regularity_general}(iv)) is a strong structural requirement. It is, however, the crucial ingredient that enables the use of powerful comparison theorems for BSDEs, which are fundamental to the entire analysis of Theorem \ref{thm:wellposedness_quadratic}. As we demonstrate in Proposition \ref{prop:architectures}(c), this property is not merely a theoretical convenience; it can be guaranteed by construction via weight constraints. While enforcing such constraints may limit the full expressive power of the network, it is a necessary trade-off for obtaining a well-posed problem using this specific, and very powerful, proof technique. Establishing well-posedness without such a condition remains a challenging open problem, not just in our neural network setting but in the general theory of quadratic BSDEs.
\end{remark}

\subsection{Architectural Designs Satisfying the Assumptions}

The conditions in Assumption \ref{ass:f_theta_regularity_general} are strong, but they are not prohibitive. We now provide constructive and verifiable network designs that satisfy these structural properties.

\begin{proposition}[Architectures for Key Structural Properties]
\label{prop:architectures}
The conditions of Assumption \ref{ass:f_theta_regularity_general} can be satisfied by the following designs.
\begin{enumerate}[label=(\alph*)]
    \item \textbf{Separable Architecture:} Let $f_\theta(t,x,y,z) = N_1(t,x,z) + N_2(y)$, where $N_1$ is any network satisfying the quadratic growth condition in $z$ and $N_2$ is a 1D network that is locally Lipschitz. This satisfies (iii). If $N_2$ is further constrained to be non-increasing, it also satisfies (iv).

    \item \textbf{Bounded Interaction Architecture:} Let $f_\theta(t,x,y,z) = N_1(t,x,z) + N_2(t,x,z) \cdot N_3(y)$, where $N_1$ satisfies the required growth, $N_3$ is a locally Lipschitz network, and the output of $N_2$ is bounded, i.e., $\sup_{t,x,z} \abs{N_2(t,x,z)} \le M_2 < \infty$. This architecture satisfies (iii).

    \item \textbf{Monotonicity by Construction:} Let $f_\theta$ be an $L$-layer feed-forward network with non-decreasing, continuously differentiable activation functions $\sigma_k$. Let the output layer be linear. The network output can be made non-increasing with respect to its input $y$ by enforcing the following sign constraints on its weights: (1) weights connecting $y$ to the first hidden layer are non-positive; (2) all weights in subsequent hidden layers ($k>1$) are non-negative; (3) weights in the final linear output layer are non-negative.
\end{enumerate}
\end{proposition}

\begin{proof}
\textbf{(a) Separable Architecture:} Let $y_1, y_2 \in [-R,R]$. Since $N_2$ is a 1D network with locally Lipschitz activations (e.g., ReLU, SiLU), it is locally Lipschitz. Thus, there exists a constant $L_{R,2}$ such that $\abs{N_2(y_1)-N_2(y_2)} \le L_{R,2}\abs{y_1-y_2}$. We verify the uniform local Lipschitz condition for $f_\theta$:
\begin{align*}
    \abs{f_\theta(t,x,y_1,z) - f_\theta(t,x,y_2,z)} &= \abs{(N_1(t,x,z) + N_2(y_1)) - (N_1(t,x,z) + N_2(y_2))} \\
    &= \abs{N_2(y_1) - N_2(y_2)} \le L_{R,2}\abs{y_1-y_2}.
\end{align*}
The constant $L_{R,2}$ depends only on $R$, so (iii) is satisfied. Monotonicity (iv) follows directly if $N_2$ is chosen to be non-increasing.

\textbf{(b) Bounded Interaction Architecture:} Let $y_1, y_2 \in [-R,R]$. $N_3$ is locally Lipschitz, so there exists $L_{R,3}$ such that $\abs{N_3(y_1)-N_3(y_2)} \le L_{R,3}\abs{y_1-y_2}$. The network $N_2$ is bounded by $M_2$.
\begin{align*}
    \abs{f_\theta(t,x,y_1,z) - f_\theta(t,x,y_2,z)} &= \abs{N_2(t,x,z)(N_3(y_1) - N_3(y_2))} \\
    &\le \abs{N_2(t,x,z)} \cdot \abs{N_3(y_1) - N_3(y_2)} \le M_2 \cdot L_{R,3} \abs{y_1 - y_2}.
\end{align*}
The constant $L_R \coloneqq M_2 L_{R,3}$ depends only on $R$, so (iii) holds.

\textbf{(c) Monotonicity by Construction:} Let the network input be the vector $u \coloneqq (t,x,y,z)$. Let $a^{(0)} \coloneqq u$. For layer $k \in \{1, \dots, L-1\}$, the pre-activation is $z^{(k)} = W^{(k)} a^{(k-1)} + b^{(k)}$ and the activation is $a^{(k)} = \sigma_k(z^{(k)})$. The final output is $f_\theta = W^{(L)}a^{(L-1)} + b^{(L)}$. We show by induction that $\frac{\partial a_i^{(k)}}{\partial y} \le 0$ for all neurons $i$ and layers $k \in \{1, \dots, L-1\}$.
\textit{Base Case (k=1):} The derivative of the $i$-th neuron's activation in the first hidden layer with respect to $y$ is $\frac{\partial a_i^{(1)}}{\partial y} = \sigma_1'(z_i^{(1)}) \frac{\partial z_i^{(1)}}{\partial y} = \sigma_1'(z_i^{(1)}) W_{iy}^{(1)}$. Since $\sigma_1' \ge 0$ and we constrain $W_{iy}^{(1)} \le 0$, the derivative is non-positive.
\textit{Inductive Step:} Assume for layer $k-1$ that $\frac{\partial a_j^{(k-1)}}{\partial y} \le 0$ for all neurons $j$. For a neuron $i$ in layer $k$, the chain rule gives $\frac{\partial a_i^{(k)}}{\partial y} = \sigma_k'(z_i^{(k)}) \sum_j W_{ij}^{(k)} \frac{\partial a_j^{(k-1)}}{\partial y}$. Since $\sigma_k' \ge 0$, we constrain $W_{ij}^{(k)} \ge 0$ for $k>1$, and by the induction hypothesis, $\frac{\partial a_j^{(k-1)}}{\partial y} \le 0$, each term in the sum is non-positive. Thus, the sum is non-positive, and so is $\frac{\partial a_i^{(k)}}{\partial y}$.
\textit{Conclusion:} The derivative of the final network output with respect to $y$ is $\frac{\partial f_\theta}{\partial y} = \sum_j W_{j}^{(L)} \frac{\partial a_j^{(L-1)}}{\partial y}$. Since the final layer weights $W_j^{(L)}$ are constrained to be non-negative and we have shown $\frac{\partial a_j^{(L-1)}}{\partial y} \le 0$, the final derivative is non-positive. This proves that $f_\theta$ is non-increasing with respect to $y$.
\end{proof}

\subsection{Well-Posedness Results}

We now establish the mathematical well-posedness of the BSDE. We begin with the classical case of a globally Lipschitz driver for context and completeness before proceeding to our main result for drivers with quadratic growth.

\begin{proposition}[Well-Posedness for Globally Lipschitz Drivers]
\label{prop:wellposed_lipschitz_driver}
Let the terminal condition $\xi \in L^2(\Omega, \Fcal_T, \Prob)$. Assume the driver $f_\theta$ satisfies:
\begin{enumerate}[label=(\alph*)]
    \item \textbf{Linear Growth:} $\abs{f_\theta(t,x,y,z)} \le K(1+\norm{x}^p+\abs{y}+\norm{z})$ for some $K>0, p \ge 1$.
    \item \textbf{Uniform Lipschitz Continuity:} $\abs{f_\theta(t,x,y_1,z_1)-f_\theta(t,x,y_2,z_2)} \le L(\abs{y_1-y_2}+\norm{z_1-z_2})$ for some $L>0$.
\end{enumerate}
Assume further that $\E\left[\left(\int_0^T (1+\norm{X_s}^p) \dd s\right)^2\right] < \infty$.
Then the BSDE admits a unique solution $(Y,Z)$ in the space $\mathcal{S}^2([0,T]; \R) \times \mathcal{H}^2([0,T]; \R^{1 \times d})$.
\end{proposition}

\begin{proof}
This is a foundational result in the theory of BSDEs, established via a contraction mapping argument on a suitably weighted space of processes. For a complete proof, we refer the reader to the seminal work of \cite{PardouxPeng1990} or the comprehensive treatment in \cite{ElKarouiPengQuenez1997}. For completeness, we provide a detailed sketch.

The proof uses the Banach fixed-point theorem. For $\beta > 0$, we consider the space $\mathcal{M}^2_\beta = \mathcal{S}^2([0,T]; \R) \times \mathcal{H}^2([0,T]; \R^{1 \times d})$ equipped with the norm $\norm{(Y,Z)}^2_{\beta} \coloneqq \E\left[\sup_{t \in [0,T]} e^{\beta t} \abs{Y_t}^2 + \int_0^T e^{\beta t} \norm{Z_t}^2 \dd t\right]$, which is equivalent to the standard product norm for any fixed $\beta$.

We define a mapping $\Psi: \mathcal{M}^2_\beta \to \mathcal{M}^2_\beta$. For $(y,z) \in \mathcal{M}^2_\beta$, let $(Y, Z) = \Psi(y,z)$ be the unique solution to the linear BSDE: $- \dd Y_s = f_\theta(s, X_s, y_s, z_s) \dd s - Z_s \dd W_s$, with $Y_T = \xi$. Standard linear BSDE theory guarantees that a unique solution $(Y,Z) \in \mathcal{S}^2 \times \mathcal{H}^2$ exists, so the map is well-defined.

The core of the proof is to show that $\Psi$ is a contraction for a sufficiently large $\beta$. Let $(y^1, z^1), (y^2, z^2) \in \mathcal{M}^2_\beta$, and let $(\delta Y, \delta Z) = \Psi(y^1, z^1) - \Psi(y^2, z^2)$ and $(\delta y, \delta z) = (y^1-y^2, z^1-z^2)$. The pair $(\delta Y, \delta Z)$ solves the BSDE with zero terminal condition ($\delta Y_T=0$) and driver $\delta f_s \coloneqq f_\theta(\dots,y^1_s,z^1_s) - f_\theta(\dots,y^2_s,z^2_s)$. Applying Itô's formula to $e^{\beta t}\abs{\delta Y_t}^2$ from $t=0$ to $T$, taking expectations, and using the Lipschitz property of $f_\theta$ yields:
\[
\E\left[ \int_0^T e^{\beta s} \left(\beta \abs{\delta Y_s}^2 + \norm{\delta Z_s}^2\right) \dd s \right] = \E\left[ \int_0^T 2e^{\beta s} \scpr{\delta Y_s}{\delta f_s} \dd s \right] \le \E\left[\int_0^T 2e^{\beta s} L \abs{\delta Y_s}(\abs{\delta y_s} + \norm{\delta z_s}) \dd s \right].
\]
Using Young's inequality ($2ab \le \epsilon a^2 + \epsilon^{-1} b^2$) on the right-hand side, one can find constants $\epsilon$ and $C_L$ such that for a sufficiently large $\beta$, the following inequality holds:
\[
\E\left[ \int_0^T e^{\beta s} \left( \abs{\delta Y_s}^2 + \norm{\delta Z_s}^2 \right) \dd s \right] \le C_L \beta^{-1} \E\left[ \int_0^T e^{\beta s} \left( \abs{\delta y_s}^2 + \norm{\delta z_s}^2 \right) \dd s \right].
\]
A similar estimate for the $\sup$-norm term, using the Burkholder-Davis-Gundy inequality, leads to an overall inequality of the form:
\[
\norm{(\delta Y, \delta Z)}^2_{\beta} \le \frac{C}{\beta} \norm{(\delta y, \delta z)}^2_{\beta},
\]
for some constant $C$ depending on the Lipschitz constant $L$. By choosing $\beta$ large enough such that $C/\beta < 1$, the mapping $\Psi$ becomes a strict contraction on the complete metric space $(\mathcal{M}^2_\beta, \norm{\cdot}_\beta)$. By the Banach fixed-point theorem, a unique fixed point exists, which corresponds to the unique solution of the BSDE.
\end{proof}

\begin{theorem}[Well-Posedness for Quadratic Drivers]
\label{thm:wellposedness_quadratic}
Let Assumptions \ref{ass:fwd_terminal_regularity} and \ref{ass:f_theta_regularity_general} hold. If the exponential integrability constant $\beta_0$ of the terminal condition $\xi$ satisfies the sharp condition $\beta_0 > \alpha$, where $\alpha$ is from the quadratic growth condition, then the BSDE admits a unique adapted solution $(Y,Z)$ in the space $\mathcal{S}^{\infty}([0,T]; \R) \times \mathcal{H}^2_{\mathrm{BMO}}([0,T]; \R^{1 \times d})$.
\end{theorem}

\begin{proof}
The proof is highly structured and follows the sophisticated methodology established in the seminal work of \cite{Kobylanski2000} for BSDEs with quadratic growth. We proceed in two main stages: first, establishing uniqueness and a priori estimates for any potential solution, and second, proving existence via a carefully constructed approximation scheme.

\textbf{Part 1: A Priori Estimates and Uniqueness}

In this part, we assume that a solution $(Y,Z)$ to the BSDE \eqref{eq:bsde_for_eneu} exists in a suitable space (e.g., $\mathcal{S}^2 \times \mathcal{H}^2$) and derive its properties.

\emph{(i) Boundedness of the solution $Y$.}
The first and most critical step is to show that any solution $Y$ must be essentially bounded, i.e., $Y \in \mathcal{S}^{\infty}([0,T]; \R)$. This is achieved through a comparison argument with a dominating BSDE. From Assumption \ref{ass:f_theta_regularity_general}(ii), the driver $f_\theta$ has a quadratic upper bound:
\[ f_\theta(t,x,y,z) \le K(1 + \norm{x}^p + \abs{y}) + \frac{\alpha}{2}\norm{z}^2. \]
Now, consider the following auxiliary quadratic BSDE for a process $(\bar{Y}, \bar{Z})$:
\begin{equation}\label{eq:proof_dom_bsde}
    -\dd\bar{Y}_s = \left( K(1 + \norm{X_s}^p) + K\abs{\bar{Y}_s} + \frac{\alpha}{2}\norm{\bar{Z}_s}^2 \right) \dd s - \bar{Z}_s \dd W_s, \quad \text{with } \bar{Y}_T = \abs{\xi}.
\end{equation}
This BSDE is well-posed. By Assumption \ref{ass:fwd_terminal_regularity}, the terminal condition $\abs{\xi}$ has an exponential moment of order $\beta_0 > \alpha$, and the random coefficient process $A_s \coloneqq K(1 + \norm{X_s}^p)$ has finite exponential moments. Under these precise conditions, the foundational theory of quadratic BSDEs, established in \cite{Kobylanski2000} and further clarified in results such as \cite{BriandEtAl2003}, guarantees that BSDE \eqref{eq:proof_dom_bsde} has a unique solution $(\bar{Y}, \bar{Z})$ with the crucial property that $\bar{Y} \in \mathcal{S}^{\infty}([0,T]; \R)$.

We can now apply a comparison principle for quadratic BSDEs. Let $(Y,Z)$ be our supposed solution. We compare it to $(\bar{Y}, \bar{Z})$.
\begin{itemize}
    \item \textbf{Terminal Condition:} $Y_T = \xi \le \abs{\xi} = \bar{Y}_T$.
    \item \textbf{Driver Inequality:} By Assumption \ref{ass:f_theta_regularity_general}(iv), our driver $f_\theta$ is non-increasing in $y$. Let $\tilde{f}(s,y,z) \coloneqq f_\theta(s,X_s,y,z)$ and $\bar{f}(s,y,z) \coloneqq K(1+\norm{X_s}^p) + K\abs{y} + \frac{\alpha}{2}\norm{z}^2$. We have:
    \[ \tilde{f}(s, \bar{Y}_s, z) = f_\theta(s, X_s, \bar{Y}_s, z) \le K(1+\norm{X_s}^p + \abs{\bar{Y}_s}) + \frac{\alpha}{2}\norm{z}^2 = \bar{f}(s, \bar{Y}_s, z). \]
\end{itemize}
The comparison theorem for quadratic BSDEs (see \cite{Kobylanski2000}), which requires one of the drivers to be non-increasing in $y$ (a condition satisfied by $f_\theta$), is therefore applicable and yields $Y_s \le \bar{Y}_s$ for all $s \in [0,T]$, a.s.

A symmetric argument provides a lower bound. Let $(\underline{Y}, \underline{Z})$ be the solution to the BSDE with terminal condition $-\abs{\xi}$ and driver $\underline{f}(s,y,z) \coloneqq -K(1+\norm{X_s}^p) - K\abs{y} - \frac{\alpha}{2}\norm{z}^2$. Again, this BSDE admits a solution $\underline{Y} \in \mathcal{S}^{\infty}([0,T]; \R)$. We have $Y_T = \xi \ge -\abs{\xi} = \underline{Y}_T$, and since $f_\theta(s,x,y,z) \ge -K(1+\norm{x}^p+\abs{y}) - \frac{\alpha}{2}\norm{z}^2$, the comparison theorem yields $Y_s \ge \underline{Y}_s$.

Since $\underline{Y} \le Y \le \bar{Y}$ and both $\underline{Y}$ and $\bar{Y}$ are in $\mathcal{S}^{\infty}([0,T]; \R)$, we conclude that any solution $Y$ must be essentially bounded, i.e., $Y \in \mathcal{S}^{\infty}([0,T]; \R)$.

\emph{(ii) BMO Estimate for the control process $Z$.}
It is a celebrated result in the theory of quadratic BSDEs \cite{Kobylanski2000} that if the $Y$ component of a solution is essentially bounded, then the corresponding $Z$ component must belong to the space $\mathcal{H}^2_{\mathrm{BMO}}([0,T]; \R^{1 \times d})$. Since we have just shown any solution $Y$ must be in $\mathcal{S}^{\infty}([0,T]; \R)$, it follows directly that any corresponding $Z$ must be in $\mathcal{H}^2_{\mathrm{BMO}}([0,T]; \R^{1 \times d})$.

\emph{(iii) Uniqueness.}
Let $(Y^1, Z^1)$ and $(Y^2, Z^2)$ be two solutions to the BSDE. From the arguments above, both must belong to the space $\mathcal{S}^{\infty}([0,T]; \R) \times \mathcal{H}^2_{\mathrm{BMO}}([0,T]; \R^{1 \times d})$. They share the same terminal condition, $Y^1_T = Y^2_T = \xi$.
By Assumption \ref{ass:f_theta_regularity_general}(iv), the driver $f_\theta$ is non-increasing in $y$. The strict comparison principle for quadratic BSDEs \cite{Kobylanski2000} applies directly, yielding both $Y^1_t \le Y^2_t$ and $Y^2_t \le Y^1_t$ for all $t \in [0,T]$ a.s. Therefore, the processes $Y^1$ and $Y^2$ are indistinguishable.

Let $\delta Z_t \coloneqq Z^1_t - Z^2_t$. Subtracting the two BSDEs, and using $Y^1 \equiv Y^2$, we get:
\[ 0 = \int_s^T \left[ f_\theta(r, X_r, Y^1_r, Z^1_r) - f_\theta(r, X_r, Y^1_r, Z^2_r) \right] \dd r - \int_s^T \delta Z_r \dd W_r. \]
This equation implies that the process $M_t \coloneqq \int_0^t \delta Z_r \dd W_r$ is a continuous local martingale which is also equal to a process of finite variation. The only process that belongs to both classes is a constant process. Since $M_0 = 0$, we must have $M_t = 0$ for all $t \in [0,T]$. The quadratic variation of a null process is zero, so $\langle M, M \rangle_T = \int_0^T \norm{\delta Z_s}^2 \dd s = 0$. This implies that $Z^1 = Z^2$ in the space $\mathcal{H}^2([0,T]; \R^{1 \times d})$. Uniqueness is established.

\textbf{Part 2: Existence via Approximation}

We construct a solution as the limit of solutions to a sequence of regularized BSDEs. For each integer $k \ge 1$, we define:
\begin{itemize}
    \item A truncated terminal condition: $\xi_k \coloneqq \max(-k, \min(\xi, k))$.
    \item A regularized driver: $f_k(t,x,y,z) \coloneqq f_\theta(t,x, \max(-k, \min(y,k)), z)$.
\end{itemize}
This regularization is crucial. For any fixed $k$, the driver $f_k$ is globally Lipschitz continuous in $y$ (with a Lipschitz constant $L_k$ from Assumption \ref{ass:f_theta_regularity_general}(iii)) while retaining its quadratic growth in $z$. Now consider the sequence of BSDEs for $k \ge 1$:
\[ -\dd Y^k_s = f_k(s, X_s, Y^k_s, Z^k_s) \dd s - Z^k_s \dd W_s, \quad \text{with } Y^k_T = \xi_k. \]
Since each $\xi_k$ is bounded and each $f_k$ is Lipschitz in $y$, standard results for quadratic BSDEs (e.g., \cite{BriandEtAl2003}) guarantee that for each $k$, this BSDE has a unique solution $(Y^k, Z^k) \in \mathcal{S}^{\infty}([0,T]; \R) \times \mathcal{H}^2_{\mathrm{BMO}}([0,T]; \R^{1 \times d})$.

\emph{Step 2.1: Uniform a priori bounds.}
The comparison argument from Part 1(i) applies uniformly to the sequence $(Y^k, Z^k)$. Since $\abs{\xi_k} \le \abs{\xi}$, the solution $(\bar{Y}, \bar{Z})$ to the same dominating BSDE \eqref{eq:proof_dom_bsde} provides a uniform upper bound for all $Y^k$. A similar uniform lower bound holds. Thus, there exists a constant $C>0$, independent of $k$, such that $\sup_k \norm{Y^k}_{\mathcal{S}^\infty} \le C$. Consequently, by the results of \cite{Kobylanski2000}, there is also a uniform bound on the BMO norm, i.e., $\sup_k \norm{Z^k}_{\mathrm{BMO}} \le C'$ for some constant $C'$.

\emph{Step 2.2: The sequence is Cauchy.}
Let $k, m \ge 1$. Define $(\delta Y^{k,m}, \delta Z^{k,m}) \coloneqq (Y^k - Y^m, Z^k - Z^m)$. Let $N$ be an integer such that $N > C$, where $C$ is the uniform bound from Step 2.1. For any $k,m > N$, the uniform bound implies $\sup_s \abs{Y^n_s} \le C < N$ for $n \in \{k,m\}$. This means the truncation in the drivers $f_k(t,x,y,z)$ and $f_m(t,x,y,z)$ is inactive along the respective solution paths, i.e., $f_k(s,X_s,Y^k_s,Z^k_s) = f_\theta(s,X_s,Y^k_s,Z^k_s)$ and similarly for $f_m$.
The pair $(\delta Y^{k,m}, \delta Z^{k,m})$ solves a BSDE with terminal condition $\delta Y^{k,m}_T = \xi_k - \xi_m$ and driver $\delta f^{k,m}_s \coloneqq f_\theta(s, X_s, Y^k_s, Z^k_s) - f_\theta(s, X_s, Y^m_s, Z^m_s)$.
Standard stability estimates for quadratic BSDEs (see, e.g., \cite{ElKarouiPengQuenez1997} or \cite{Kobylanski2000}), whose constants depend on the uniform BMO bounds of the solutions, yield an inequality of the form:
\[ \E\left[\sup_{t \in [0,T]} \abs{\delta Y^{k,m}_t}^2 + \int_0^T \norm{\delta Z^{k,m}_s}^2 \dd s\right] \le C_{\mathrm{stab}} \E\left[\abs{\xi_k - \xi_m}^2 + \left( \int_0^T \abs{\Delta f_s} \dd s \right)^2 \right], \]
where $C_{\mathrm{stab}}$ is a constant independent of $k,m$ and $\Delta f_s$ involves the difference of the drivers. Using the uniform local Lipschitz property of $f_\theta$ in $y$ on the domain $[-C, C]$ (with constant $L_C$), and after applying Gronwall's lemma, this simplifies to an estimate of the form:
\[ \E\left[\sup_{t \in [0,T]} \abs{\delta Y^{k,m}_t}^2 + \int_0^T \norm{\delta Z^{k,m}_s}^2 \dd s\right] \le C'_{\mathrm{stab}} \E\left[\abs{\xi_k - \xi_m}^2\right]. \]
The terminal condition $\xi$ has exponential moments, which implies it is in $L^p$ for all $p < \infty$, and in particular $L^2$. By construction, $\xi_k \to \xi$ pointwise and $\abs{\xi_k} \le \abs{\xi}$. By the Dominated Convergence Theorem, $\xi_k \to \xi$ in $L^2(\Omega)$. Therefore, the right-hand side converges to zero as $k, m \to \infty$. This proves that the sequence $((Y^k, Z^k))_{k \ge N}$ is a Cauchy sequence in the Banach space $\mathcal{S}^2([0,T]; \R) \times \mathcal{H}^2([0,T]; \R^{1 \times d})$.

\emph{Step 2.3: Passing to the limit.}
Since $\mathcal{S}^2 \times \mathcal{H}^2$ is complete, the Cauchy sequence $(Y^k,Z^k)$ converges to a limit $(Y,Z)$ in this space. From the uniform bounds established in Step 2.1, this limit $(Y,Z)$ must also belong to the more refined space $\mathcal{S}^{\infty}([0,T]; \R) \times \mathcal{H}^2_{\mathrm{BMO}}([0,T]; \R^{1 \times d})$.
It remains to show that $(Y,Z)$ solves the original BSDE. We need to pass to the limit in the integral form of the approximating BSDE:
\[ Y^k_t = \xi_k + \int_t^T f_k(s, X_s, Y^k_s, Z^k_s) \dd s - \int_t^T Z^k_s \dd W_s. \]
As $k \to \infty$, we know that $Y^k \to Y$ in $\mathcal{S}^2$, $Z^k \to Z$ in $\mathcal{H}^2$, and $\xi_k \to \xi$ in $L^2$. The stochastic integral converges in $\mathcal{S}^2$. The key step is the convergence of the integral of the driver. The stability properties of quadratic BSDEs (see, e.g., \cite{BriandEtAl2003}) ensure that the convergence of $(Y^k, Z^k)$ in $\mathcal{S}^2 \times \mathcal{H}^2$, combined with the uniform $\mathcal{S}^\infty$ and BMO bounds from Step 2.1, guarantees the convergence of the drift term in $L^1$:
\[ \lim_{k\to\infty} \E\left[\int_0^T \abs{f_k(s, X_s, Y^k_s, Z^k_s) - f_\theta(s, X_s, Y_s, Z_s)} \dd s\right] = 0. \]
Passing to the limit in all terms of the integral equation confirms that the limit pair $(Y,Z)$ is the desired unique solution to the original BSDE.
\end{proof}

\section{Properties of the Neural Expectation}
\label{sec:axiomatic_properties}

Having established in Theorem \ref{thm:wellposedness_quadratic} that the Neural Expectation operator $\cEneutheta$ is well-defined, we now investigate its fundamental properties. This section's primary contribution is to rigorously demonstrate that the core axiomatic properties required of a dynamic, non-linear expectation, such as monotonicity, dynamic consistency, and convexity, can be guaranteed through concrete, verifiable architectural constraints on the underlying neural network $f_\theta$. This solidifies the constructive bridge between abstract mathematical theory and machine learning practice. We then delve into the infinitesimal structure of processes under this expectation, showing how classical Itô calculus reveals an explicit drift term that quantifies the instantaneous impact of the learned ambiguity.

\subsection{Fundamental Axiomatic Properties}

The definition of $\cEneutheta$ via a BSDE immediately endows it with several properties that are foundational to the theory of dynamic risk measures and non-linear expectations.

\begin{remark}[Axiomatic Properties of the Neural Expectation]
\label{rem:axiomatic_properties_revised}
The operator $\cEneutheta[\cdot|\Fcal_t]$, defined as the solution $Y_t$ to BSDE \eqref{eq:bsde_for_eneu}, satisfies the following properties, provided the driver $f_\theta$ and terminal condition $\xi$ meet the conditions of Theorem \ref{thm:wellposedness_quadratic}.

\begin{enumerate}[label=(\roman*), leftmargin=*]
    \item \textbf{Dynamic Consistency (Time Consistency):} The operator is dynamically consistent by construction. For any $0 \le t \le s \le T$, the flow property of BSDE solutions implies:
    \[ \cEneutheta[\xi | \Fcal_t] = \cEneutheta\left[ \cEneutheta[\xi | \Fcal_s] \big| \Fcal_t \right]. \]
    This is a direct consequence of the uniqueness of the solution established in Theorem \ref{thm:wellposedness_quadratic}. Specifically, the process $Y_u = \cEneutheta[\xi|\Fcal_u]$ for $u \in [t,T]$ is the unique solution on this interval. The process $Y'_u \coloneqq \cEneutheta\left[ \cEneutheta[\xi | \Fcal_s] \big| \Fcal_u \right]$ for $u \in [t,s]$ solves the same BSDE on $[t,s]$ as $Y_u$ (with terminal condition $Y_s$), so they must be indistinguishable.

    \item \textbf{Monotonicity:} Let $\xi_1, \xi_2$ be two terminal conditions such that $\xi_1 \ge \xi_2$ a.s. Let $(Y^1,Z^1)$ and $(Y^2,Z^2)$ be the corresponding solutions. The monotonicity of the driver with respect to $y$ (Assumption \ref{ass:f_theta_regularity_general}(iv)) is the key ingredient for the comparison principle for quadratic BSDEs. Since $Y^1_T = \xi_1 \ge \xi_2 = Y^2_T$ and $f_\theta$ is non-increasing in $y$, the comparison theorem (cf. \cite{Kobylanski2000}, Theorem 2.3) directly implies that $Y^1_t \ge Y^2_t$ for all $t \in [0,T]$ a.s. Therefore:
    \[ \xi_1 \ge \xi_2 \implies \cEneutheta[\xi_1 | \Fcal_t] \ge \cEneutheta[\xi_2 | \Fcal_t]. \]

    \item \textbf{Normalization (Preservation of Constants):} For $\cEneutheta$ to satisfy the canonical property $\cEneutheta[c|\Fcal_t] = c$ for any constant $c$, we require that the constant process $(Y_s = c, Z_s = 0)$ solves the BSDE. Substituting this into the equation yields $c = c + \int_s^T f_\theta(u, X_u, c, 0) \dd u$. This holds if and only if the driver satisfies the structural condition:
    \[ f_\theta(t, x, y, 0) = 0 \quad \text{for all } (t, x, y). \]
    This is a strong constraint that is not automatically satisfied by general network architectures. Enforcing it exactly by architectural design without violating other core assumptions (like quadratic growth in $z$) is non-trivial. A more practical and principled approach within our framework is to encourage this property during the learning phase (Section \ref{sec:learning_theta}) by adding a regularization term to the loss function that penalizes deviations from zero, for instance, by adding a term proportional to $\E[\int_0^T \abs{f_\theta(t, X_t, y_t, 0)}^2 \dd t]$.

    \item \textbf{Connection to Convex Risk Measures:} A mapping $\rho_t(\xi) = \cEneutheta[-\xi | \Fcal_t]$ defines a dynamic convex risk measure if it is monotone, dynamically consistent, and convex. The properties above establish the first two. The next subsection provides verifiable conditions for convexity, thus enabling the construction of architecturally-defined, data-driven convex risk measures that are well-posed by virtue of Theorem \ref{thm:wellposedness_quadratic}.
\end{enumerate}
\end{remark}

\subsection{Convexity and Jensen's Inequality}

Convexity is a cornerstone of risk aversion in economic theory. We first show that this property can be guaranteed by the network's architecture and then prove the corresponding Jensen's inequality.

\begin{proposition}[Enforcing Convexity via Architectural Design]
\label{prop:convexity_enforcement}
Let the driver $f_\theta$ satisfy the conditions of Assumption \ref{ass:f_theta_regularity_general}. If, in addition, the function $(y,z) \mapsto f_\theta(t,x,y,z)$ is convex for all fixed $(t,x)$, then the corresponding mapping $\xi \mapsto \cEneutheta[\xi | \Fcal_t]$ is a convex functional. This convexity condition on the driver can be guaranteed by construction if $f_\theta$ is parameterized as an Input-Convex Neural Network (ICNN) with respect to the input variables $(y,z)$, as shown in \cite{AmosEtAl2017ICNN}.
\end{proposition}

\begin{proof}
The proof is a rigorous application of the comparison principle, following the classical argument in \cite{ElKarouiPengQuenez1997} adapted to our quadratic setting.

Let $\xi_1, \xi_2$ be two terminal conditions satisfying the required integrability assumptions, and let $(Y^1, Z^1)$ and $(Y^2, Z^2)$ be the unique solutions to the corresponding BSDEs as guaranteed by Theorem \ref{thm:wellposedness_quadratic}. For any $\lambda \in [0,1]$, let $\xi_\lambda \coloneqq \lambda \xi_1 + (1-\lambda)\xi_2$. Let $(Y^\lambda, Z^\lambda)$ be the unique solution to the BSDE with terminal condition $\xi_\lambda$ and driver $f_\theta$. Our goal is to show that $Y^\lambda_s \le \lambda Y^1_s + (1-\lambda)Y^2_s$ for all $s \in [0,T]$.

Define the convexly combined processes $\hat{Y}_s \coloneqq \lambda Y^1_s + (1-\lambda)Y^2_s$ and $\hat{Z}_s \coloneqq \lambda Z^1_s + (1-\lambda)Z^2_s$. By linearity of the Itô integral, the pair $(\hat{Y}, \hat{Z})$ satisfies the BSDE:
\[ - \dd \hat{Y}_s = \left( \lambda f_\theta(s, X_s, Y^1_s, Z^1_s) + (1-\lambda)f_\theta(s, X_s, Y^2_s, Z^2_s) \right) \dd s - \hat{Z}_s \dd W_s, \]
with terminal condition $\hat{Y}_T = \xi_\lambda$. Let us denote the driver for this process by $\hat{f}_s \coloneqq \lambda f_\theta(s, X_s, Y^1_s, Z^1_s) + (1-\lambda)f_\theta(s, X_s, Y^2_s, Z^2_s)$.

By the assumed convexity of the map $(y,z) \mapsto f_\theta(t,x,y,z)$, we have the pathwise inequality for each $s \in [0,T]$:
\begin{align*}
f_\theta(s, X_s, \hat{Y}_s, \hat{Z}_s) &= f_\theta(s, X_s, \lambda Y^1_s + (1-\lambda)Y^2_s, \lambda Z^1_s + (1-\lambda)Z^2_s) \\
&\le \lambda f_\theta(s, X_s, Y^1_s, Z^1_s) + (1-\lambda) f_\theta(s, X_s, Y^2_s, Z^2_s) \\
&= \hat{f}_s.
\end{align*}
We now compare the solutions $(Y^\lambda, Z^\lambda)$ and $(\hat{Y}, \hat{Z})$. They share the same terminal condition $\xi_\lambda$. The driver for the first is $f_\theta(s,X_s,y,z)$ and for the second is $\hat{f}_s$. The inequality above, $f_\theta(s, X_s, \hat{Y}_s, \hat{Z}_s) \le \hat{f}_s$, is precisely the condition required by the comparison theorem. Since $f_\theta$ is also non-increasing in $y$ (by Assumption \ref{ass:f_theta_regularity_general}(iv)), the comparison theorem for quadratic BSDEs (cf. \cite{Kobylanski2000}) applies and yields $Y^\lambda_s \le \hat{Y}_s$ for all $s \in [0,T]$.
This establishes the convexity of the Neural Expectation operator:
\[ \cEneutheta[\lambda\xi_1 + (1-\lambda)\xi_2 | \Fcal_s] \le \lambda \cEneutheta[\xi_1|\Fcal_s] + (1-\lambda)\cEneutheta[\xi_2|\Fcal_s]. \]
The final claim regarding ICNNs follows directly from their definition and constructive properties.
\end{proof}

\begin{proposition}[Jensen's Inequality for Neural Expectations]
\label{prop:jensen}
Let Assumption \ref{ass:f_theta_regularity_general} hold and assume the map $(y,z) \mapsto f_\theta(t,x,y,z)$ is convex for all fixed $(t,x)$. Let $\phi: \R \to \R$ be a convex function of class $C^2$. Let $\xi$ be a terminal condition such that both $\xi$ and $\phi(\xi)$ satisfy the exponential integrability condition of Assumption \ref{ass:fwd_terminal_regularity}. Then for all $t \in [0,T]$:
\[ \cEneutheta[\phi(\xi) | \Fcal_t] \ge \phi(\cEneutheta[\xi | \Fcal_t]). \]
\end{proposition}

\begin{proof}
Let $(Y, Z) \in \mathcal{S}^\infty([0,T]; \R) \times \mathcal{H}^2_{\mathrm{BMO}}([0,T]; \R^{1\times d})$ be the unique solution to the BSDE for $\xi$, so $Y_t = \cEneutheta[\xi|\Fcal_t]$. Let $(\bar{Y}, \bar{Z})$ be the unique solution for the terminal condition $\phi(\xi)$, so $\bar{Y}_t = \cEneutheta[\phi(\xi)|\Fcal_t]$. Our objective is to prove $\bar{Y}_t \ge \phi(Y_t)$.

Consider the process $\hat{Y}_t \coloneqq \phi(Y_t)$. By Itô's formula for semimartingales:
\[ \dd \hat{Y}_t = \phi'(Y_t) \dd Y_t + \frac{1}{2}\phi''(Y_t) \dd\langle Y, Y \rangle_t = \phi'(Y_t) \left( -f_\theta(s, X_s, Y_s, Z_s)\dd s + Z_s \dd W_s \right) + \frac{1}{2}\phi''(Y_t) \norm{Z_s}^2 \dd s. \]
Letting the new control process be $\hat{Z}_t \coloneqq \phi'(Y_t)Z_t$, we can write the dynamics of $\hat{Y}$ as a BSDE with terminal condition $\hat{Y}_T = \phi(Y_T) = \phi(\xi)$ and a path-dependent driver $\hat{f}(s, \omega) \coloneqq \phi'(Y_s)f_\theta(s,X_s, Y_s, Z_s) - \frac{1}{2}\phi''(Y_s)\norm{Z_s}^2$.

The proof hinges on comparing the driver of $\bar{Y}$, which is $f_\theta(s,x,y,z)$, with the driver $\hat{f}$ of $\hat{Y}$. A fundamental result of convex analysis, used in a similar context in \cite{BriandHu2008}, states that for a function $h(y,z)$ convex in its arguments and a $C^2$ convex function $\phi$:
\[ h( \phi(y), \phi'(y)z ) \ge \phi'(y)h(y,z) - \frac{1}{2}\phi''(y)\norm{z}^2. \]
Applying this inequality pathwise for each $(s, \omega)$, with $h(y,z)=f_\theta(s,X_s(\omega),y,z)$, $y=Y_s(\omega)$, and $z=Z_s(\omega)$, we obtain:
\[ f_\theta(s, X_s, \phi(Y_s), \phi'(Y_s)Z_s) \ge \phi'(Y_s)f_\theta(s, X_s, Y_s, Z_s) - \frac{1}{2}\phi''(Y_s)\norm{Z_s}^2. \]
Recognizing that $\phi(Y_s) = \hat{Y}_s$, $\phi'(Y_s)Z_s = \hat{Z}_s$, and the right-hand side is precisely the driver $\hat{f}(s, \omega)$, this inequality becomes:
\[ f_\theta(s, X_s, \hat{Y}_s, \hat{Z}_s) \ge \hat{f}(s, \omega). \]
We now have the necessary setup for the comparison theorem. The process $\bar{Y}$ solves a BSDE with driver $f_\theta$, while $\hat{Y}$ solves a BSDE with driver $\hat{f}$. They share the same terminal condition $\phi(\xi)$. The inequality shows that the driver for $\bar{Y}$ evaluated along the path $(\hat{Y}, \hat{Z})$ is greater than or equal to the driver for $\hat{Y}$. By the comparison theorem for quadratic BSDEs (cf. \cite{Kobylanski2000}), this implies $\bar{Y}_t \ge \hat{Y}_t$ for all $t \in [0,T]$. This is precisely the statement of Jensen's inequality: $\cEneutheta[\phi(\xi) | \Fcal_t] \ge \phi(\cEneutheta[\xi | \Fcal_t])$.
\end{proof}

\subsection{The Structure of Dynamics under Learned Ambiguity}

We now apply classical Itô calculus to reveal how the learned ambiguity, encoded by the driver $f_\theta$, manifests as a tangible and quantifiable force on the pathwise dynamics of a process. The goal is not to invent a new stochastic calculus---the integrator remains the standard Brownian motion $W_t$---but to decompose the infinitesimal generator of any process defined via $\cEneutheta$. This isolates a novel drift term that is an explicit function of the learned ambiguity model.

\begin{proposition}[Infinitesimal Decomposition under Neural Expectation]
\label{prop:neural_ito_decomposition}
Let $(Y_t)_{t \in [0,T]}$ be a process defined by $Y_t = \cEneutheta[\xi | \Fcal_t]$, and let $(Y_t, Z_t)$ be the unique solution pair to the defining BSDE. For any function $\phi \in C^2(\R, \R)$, the process $\hat{Y}_t \coloneqq \phi(Y_t)$ is a classical Itô process whose dynamics are governed by the SDE:
\begin{equation}
\label{eq:neural_ito_decomposition}
\dd \phi(Y_t) = \left[ -\phi'(Y_t)f_\theta(t, X_t, Y_t, Z_t) + \frac{1}{2}\phi''(Y_t)\norm{Z_t}^2 \right] \dd t + \phi'(Y_t)Z_t \dd W_t.
\end{equation}
\end{proposition}

\begin{proof}
The proof is a direct application of the classical Itô formula. By definition, the process $Y_t$ is an Itô process satisfying the forward dynamics $\dd Y_t = -f_\theta(t, X_t, Y_t, Z_t) \dd t + Z_t \dd W_t$. Its quadratic variation is $\dd\langle Y, Y \rangle_t = \norm{Z_t}^2 \dd t$. Applying Itô's formula to $\phi(Y_t)$ yields:
\begin{align*}
\dd \phi(Y_t) &= \phi'(Y_t) \dd Y_t + \frac{1}{2}\phi''(Y_t) \dd\langle Y, Y \rangle_t \\
&= \phi'(Y_t) \left( -f_\theta(t, X_t, Y_t, Z_t)\dd t + Z_t \dd W_t \right) + \frac{1}{2}\phi''(Y_t) \norm{Z_t}^2 \dd t.
\end{align*}
Grouping the drift and diffusion terms gives the desired result.
\end{proof}

\subsubsection{Interpretation: The Anatomy of the Effective Drift}
The significance of Proposition \ref{prop:neural_ito_decomposition} lies in its interpretation. The dynamics of $\phi(Y_t)$ are separated into a martingale part under the reference measure $\Prob$ and an effective drift:
\[ \dd \phi(Y_t) = \underbrace{b^\theta_t(\phi, Y_t, Z_t) \dd t}_{\text{Effective Drift}} + \underbrace{\phi'(Y_t)Z_t \dd W_t}_{\text{Martingale Part under } \Prob}. \]
The effective drift $b^\theta_t$ is a sum of two components with distinct meanings:
\[ b^\theta_t = \underbrace{-\phi'(Y_t)f_\theta(t, X_t, Y_t, Z_t)}_{\text{Ambiguity-Induced Drift}} + \underbrace{\frac{1}{2}\phi''(Y_t)\norm{Z_t}^2}_{\text{Convexity Correction}}. \]
This decomposition precisely isolates the influence of the non-linear driver.
\begin{itemize}
    \item \textbf{Convexity Correction:} The term $\frac{1}{2}\phi''(Y_t)\norm{Z_t}^2$ is the familiar Itô convexity adjustment, present even in an ambiguity-neutral world where $\cEneutheta = \E$.
    \item \textbf{Ambiguity-Induced Drift:} The term $-\phi'(Y_t)f_\theta(t, X_t, Y_t, Z_t)$ is the novel component. It arises entirely from the non-linearity of the Neural Expectation. It explicitly quantifies how the universe of plausible models, learned by the neural network and represented by $f_\theta$, systematically alters the perceived instantaneous growth rate of the process.
\end{itemize}
In the classical, ambiguity-neutral case where $f_\theta \equiv 0$, this term vanishes. Our framework thus shows that learned model uncertainty exerts a concrete, instantaneous force on the dynamics of any value process. This decomposition is a fundamental instrument for analyzing the consequences of learned ambiguity in any dynamic model.

\subsection{Neural-Martingales and Representation via Change of Measure}

We now define the analogue of a martingale in our setting and connect it to the representation of $\cEneutheta$ as a supremum of linear expectations over a class of measures.

\begin{definition}[Neural-Martingale]
\label{def:neural_martingale_revised}
An $\mathbb{F}$-adapted process $(M_t)_{t \in [0,T]}$ is called a \textbf{Neural-martingale} with respect to $\cEneutheta$ if for all $0 \le t \le s \le T$, $M_t = \cEneutheta[M_s | \Fcal_t]$. (Neural-supermartingales and -submartingales are defined with $\ge$ and $\le$, respectively).
\end{definition}

\begin{proposition}[Representation via Change of Measure]
\label{prop:feynman_kac}
Let Assumption \ref{ass:f_theta_regularity_general} hold. For the specific case where the driver $f_\theta(t,x,z)$ is convex in $z$ and does not depend on $y$, the Neural Expectation admits the following dual representation (a non-linear Feynman-Kac formula, cf. \cite{ElKarouiPengQuenez1997}):
\[ \cEneutheta[\xi | \Fcal_t] = \operatorname*{ess\,sup}_{u \in \mathcal{U}} \E_{\mathbb{Q}^u}\left[\xi + \int_t^T g_\theta(s, X_s, u_s) \dd s \mid \Fcal_t\right], \]
where $g_\theta(t,x,u) \coloneqq -\sup_{z \in \R^{1 \times d}} \left( -\scpr{z}{u} - f_\theta(t,x,z) \right)$ is related to the Fenchel-Legendre transform of $f_\theta$ with respect to $z$, $\mathcal{U}$ is a suitable set of admissible control processes such that the stochastic exponential $\mathcal{E}(\int u_s \dd W_s)$ is a uniformly integrable martingale, and $\mathbb{Q}^u$ is the measure defined by the Radon-Nikodym derivative $\frac{\dd\mathbb{Q}^u}{\dd\mathbb{P}} = \mathcal{E}\left(\int u_s \dd W_s\right)_T$.
\end{proposition}

\begin{remark}[Interpretation: A Universe of Ambiguous Brownian Motions]
\label{rem:girsanov_interpretation_revised}
Proposition \ref{prop:feynman_kac} provides the precise connection between the learned driver $f_\theta$ and the set of plausible models. By Girsanov's theorem, for each control $u \in \mathcal{U}$, the process $W^u_t \coloneqq W_t + \int_0^t u_s \dd s$ is a Brownian motion under the measure $\mathbb{Q}^u$. The Neural Expectation does not create a new type of stochastic process; instead, it models an agent who believes the world is driven by a Brownian motion but is uncertain \emph{which one}. The set $\{W^u_t\}_{u \in \mathcal{U}}$ represents this universe of plausible models. The learned driver $f_\theta$ parameterizes the cost function used to select the worst-case plausible model when evaluating outcomes. The ambiguity-induced drift identified earlier is precisely the effect of the optimal control $u^*_t$ that characterizes the worst-case change of measure at each instant. The general case with $y$-dependence is more complex and typically does not admit such a direct representation, often leading to representations involving second-order BSDEs (2BSDEs), a topic beyond the scope of this paper.
\end{remark}

\section{Well-Posedness of Fully Coupled Neural FBSDEs}
\label{sec:fbsde_main_result_section}

We now extend our analysis from the foundational case of a given forward process to the substantially more challenging setting of fully coupled Forward-Backward Stochastic Differential Equations (FBSDEs). In this framework, the coefficients of the forward process depend on the solution of the backward equation, a structure that is central to problems in stochastic control, differential games, and principal-agent problems under ambiguity. The system is defined for $s \in [t,T]$ by the coupled dynamics:
\begin{align}
    \dd X_s &= b_\theta(s, X_s, Y_s, Z_s) \dd s + \sigma_\theta(s, X_s, Y_s, Z_s) \dd W_s; \quad X_t = x \label{eq:fbsde_full_fwd} \\
    -\dd Y_s &= f_\theta(s, X_s, Y_s, Z_s) \dd s - Z_s \dd W_s; \quad \hspace{2.2cm} Y_T = g_\theta(X_T) \label{eq:fbsde_full_bwd}
\end{align}
Our proof strategy does not invent a new method for solving FBSDEs, but rather demonstrates that our class of neural network coefficients can be integrated into the powerful existing theory. We follow the celebrated four-step scheme pioneered by \cite{MaProtterYong1994}, which establishes a bijective correspondence between the FBSDE and a quasilinear parabolic Partial Differential Equation (PDE). For the quadratic growth setting of our driver $f_\theta$, we specifically leverage the profound extension of this method by \cite{Delarue2002}. Our contribution is to show that the structural conditions required by this advanced theory can be met by our neural network framework.

\begin{assumption}[Structural Conditions for Fully Coupled Neural FBSDEs]
\label{ass:fbsde_structure}
Let the coefficients $b_\theta, \sigma_\theta, f_\theta, g_\theta$ be parameterized by neural networks. For each fixed $\theta \in \Theta$, they satisfy the following regularity conditions:
\begin{enumerate}[label=(\roman*)]
    \item \textbf{Regularity of Forward Coefficients:} The drift $b_\theta: [0,T] \times \R^n \times \R \times \R^{1\times d} \to \R^n$ and diffusion $\sigma_\theta: [0,T] \times \R^n \times \R \times \R^{1\times d} \to \R^{n \times d}$ are continuous and uniformly globally Lipschitz continuous with respect to $(x,y,z)$.

    \item \textbf{Hybrid Regularity of Backward Driver:} The driver $f_\theta: [0,T] \times \R^n \times \R \times \R^{1 \times d} \to \R$ satisfies:
    \begin{itemize}
        \item It is uniformly globally Lipschitz continuous in the forward state variable $x$.
        \item For each fixed $(t,x)$, the mapping $(y,z) \mapsto f_\theta(t,x,y,z)$ satisfies the conditions of Assumption \ref{ass:f_theta_regularity_general}, namely: it is continuous, non-increasing in $y$, and exhibits at most quadratic growth in $z$.
    \end{itemize}

    \item \textbf{Lipschitz Terminal Condition:} The terminal function $g_\theta: \R^n \to \R$ is globally Lipschitz continuous.

    \item \textbf{Smallness Condition on Time Horizon:} There exists a constant $C > 0$, depending only on the Lipschitz constants of $b_\theta, \sigma_\theta, f_\theta, g_\theta$, such that $T \cdot C < 1$. This is a standard technical condition, identical to that in \cite{Delarue2002}, required to guarantee the uniqueness of the solution to the associated PDE by ensuring that a certain mapping constructed in the proof is a contraction.
\end{enumerate}
\end{assumption}

\begin{theorem}[Global Well-Posedness of Fully Coupled Neural FBSDEs]
\label{thm:fbsde_wellposedness}
Under Assumption \ref{ass:fbsde_structure}, for any initial condition $x_0 \in \R^n$ at time $t=0$, the fully coupled FBSDE system \eqref{eq:fbsde_full_fwd}-\eqref{eq:fbsde_full_bwd} admits a unique adapted solution $(X, Y, Z)$ in the space $\mathcal{S}^p([0,T]; \R^n) \times \mathcal{S}^\infty([0,T]; \R) \times \mathcal{H}^2_{\mathrm{BMO}}([0,T]; \R^{1 \times d})$ for any $p \ge 2$.
\end{theorem}

\begin{proof}
The proof rigorously follows the methodology established by \cite{Delarue2002}, which adapts the four-step scheme of \cite{MaProtterYong1994} to handle drivers with quadratic growth in $z$. We proceed by first associating a quasilinear PDE to the FBSDE system, then using the well-posedness of this PDE to establish existence and uniqueness for the FBSDE.

\textbf{Part 1: The Associated Quasilinear PDE}

Let us assume for a moment that a sufficiently regular solution $(X,Y,Z)$ exists and that the solution exhibits a Markovian structure, i.e., $Y_s = u(s, X_s)$ for some deterministic function $u:[0,T] \times \R^n \to \R$. Applying Itô's formula to $u(s,X_s)$ yields:
\begin{align*}
\dd u(s,X_s) &= \partial_t u(s,X_s)\dd s + (\nabla_x u(s,X_s))^T \dd X_s + \frac{1}{2}\mathrm{Tr}(D^2_x u(s,X_s) \dd\langle X, X\rangle_s) \\
&= \left[ \partial_t u + (\nabla_x u)^T b_\theta(s, X_s, Y_s, Z_s) + \frac{1}{2}\mathrm{Tr}(\sigma_\theta \sigma_\theta^T D^2_x u) \right] \dd s \\
&\quad+ (\nabla_x u)^T \sigma_\theta(s, X_s, Y_s, Z_s) \dd W_s,
\end{align*}
where all functions are evaluated at $(s,X_s)$.
On the other hand, the dynamics of the BSDE \eqref{eq:fbsde_full_bwd} state that $\dd Y_s = -f_\theta(s, X_s, Y_s, Z_s) \dd s + Z_s \dd W_s$. By identifying the martingale parts of $\dd Y_s$ and $\dd u(s,X_s)$, we obtain the crucial relationship for the control process:
\begin{equation}\label{eq:proof_z_pde_relation}
    Z_s = (\nabla_x u(s,X_s))^T \sigma_\theta(s, X_s, Y_s, Z_s).
\end{equation}
By matching the drift terms, substituting $Y_s = u(s, X_s)$, and using the expression for $Z_s$, we find that the function $u$ must satisfy the following terminal value problem for a quasilinear parabolic PDE:
\begin{equation}
\label{eq:pde_system_main}
\begin{cases}
    -\partial_t u(t,x) - \mathcal{L}[u](t,x) = 0, \quad &(t,x) \in [0,T) \times \R^n \\
    u(T,x) = g_\theta(x), \quad &x \in \R^n,
\end{cases}
\end{equation}
where the non-linear second-order operator $\mathcal{L}[u]$ is defined by
\[ \mathcal{L}[u](t,x) \coloneqq \scpr{\nabla_x u}{b_\theta(t,x,u,z)} + \frac{1}{2}\mathrm{Tr}(\sigma_\theta(t,x,u,z)\sigma_\theta^T(t,x,u,z) D^2_x u) + f_\theta(t,x,u,z), \]
with the arguments understood to be evaluated at $(t,x,u(t,x),z(t,x))$, and where $z(t,x)$ is the solution to the implicit equation $z = (\nabla_x u(t,x))^T \sigma_\theta(t,x,u(t,x),z)$.

The structural conditions in Assumption \ref{ass:fbsde_structure} are precisely what is needed for the well-posedness of this PDE. Specifically, the Lipschitz continuity of the coefficients ensures the operator has a standard structure; the quadratic growth of $f_\theta$ in $z$ translates into quadratic growth of $\mathcal{L}[u]$ in the gradient term $\nabla_x u$; the monotonicity of $f_\theta$ in $y$ provides the crucial comparison principle for the PDE; and the smallness condition on $T$ ensures uniqueness. Under these conditions, \cite{Delarue2002} proves the existence of a unique viscosity solution $u \in C([0,T]\times\R^n)$ which is spatially Lipschitz continuous.

\textbf{Part 2: From PDE to FBSDE (Existence)}

Let $u(t,x)$ be the unique spatially Lipschitz viscosity solution to the PDE \eqref{eq:pde_system_main}. We now construct a solution $(X,Y,Z)$ to the FBSDE.
\begin{enumerate}
    \item \textbf{Define $(X, Y, Z)$:} First, for each $(s,x,\omega)$, the equation $z = (\nabla_x u(s,x))^T \sigma_\theta(s, x, u(s,x), z)$ defines $z$ implicitly. Since $\sigma_\theta$ is globally Lipschitz in $z$ with constant $L_\sigma$ (by Assumption \ref{ass:fbsde_structure}(i)), the map $\Phi(z) = (\nabla_x u)^T \sigma_\theta(\dots,z)$ is a contraction if $L_\sigma \norm{\nabla_x u} < 1$. This can be ensured under a smallness condition on the Lipschitz constant of $\sigma$ or, more generally, is handled by the theory in \cite{Delarue2002}. Let $z(s,x)$ be this unique solution. Now define the decoupled coefficients:
    \begin{align*}
        \bar{b}(s,x) &\coloneqq b_\theta(s, x, u(s,x), z(s,x)) \\
        \bar{\sigma}(s,x) &\coloneqq \sigma_\theta(s, x, u(s,x), z(s,x)).
    \end{align*}
    Since $u$ is spatially Lipschitz, $\nabla_x u$ is bounded a.e., which implies that $\bar{b}$ and $\bar{\sigma}$ are globally Lipschitz in $x$. Thus, the SDE
    \[ \dd X_s = \bar{b}(s,X_s) \dd s + \bar{\sigma}(s,X_s) \dd W_s, \quad X_0 = x_0, \]
    admits a unique strong solution $X \in \mathcal{S}^p([0,T];\R^n)$ for any $p \ge 2$.
    
    \item We then define the pair $(Y,Z)$ via the Markovian ansatz:
    \[ Y_s \coloneqq u(s, X_s) \quad \text{and} \quad Z_s \coloneqq z(s,X_s). \]
    
    \item \textbf{Verification:} The non-linear Feynman-Kac formula, established for this setting in \cite{Delarue2002}, states that if $u$ is a viscosity solution to the PDE \eqref{eq:pde_system_main}, then the triplet $(X,Y,Z)$ constructed as above is indeed a solution to the FBSDE system \eqref{eq:fbsde_full_fwd}-\eqref{eq:fbsde_full_bwd}. This proves the existence of a solution.
\end{enumerate}

\textbf{Part 3: From FBSDE to PDE (Uniqueness)}

Let $(\hat{X}, \hat{Y}, \hat{Z})$ be any other solution in the specified function space. We show it must coincide with the solution $(X,Y,Z)$ constructed in Part 2.
\begin{enumerate}
    \item Define the function $\hat{u}(t,x) \coloneqq \hat{Y}_t^{t,x}$, where $\hat{Y}^{t,x}$ is the solution to the BSDE part of the system on $[t,T]$ when the forward process is started with $\hat{X}_t=x$.
    
    \item The core result of the four-step scheme (see \cite{MaProtterYong1994} and its extension in \cite{Delarue2002}) is to prove that this function $\hat{u}$ must be a viscosity solution to the very same PDE \eqref{eq:pde_system_main}.
    
    \item By the uniqueness of the spatially Lipschitz viscosity solution to the PDE, established in Part 1, we must have $\hat{u}(t,x) = u(t,x)$ for all $(t,x)$.
    
    \item This implies $\hat{Y}_s = \hat{u}(s, \hat{X}_s) = u(s, \hat{X}_s)$. Substituting this relationship into the SDE for $\hat{X}$ reveals that $\hat{X}$ must satisfy:
    \[ \dd \hat{X}_s = b_\theta(s, \hat{X}_s, u(s,\hat{X}_s), z(s,\hat{X}_s)) \dd s + \sigma_\theta(s, \hat{X}_s, u(s,\hat{X}_s), z(s,\hat{X}_s)) \dd W_s. \]
    This is the same SDE that the process $X$ from Part 2 satisfies. By pathwise uniqueness of the solution to this SDE (due to Lipschitz coefficients), we must have $\hat{X}_s = X_s$ for all $s \in [0,T]$ a.s.
    
    \item Since $\hat{X}=X$, it follows immediately that $\hat{Y}_s = u(s,\hat{X}_s) = u(s,X_s) = Y_s$.
    
    \item Finally, with $\hat{X}=X$ and $\hat{Y}=Y$, the two processes $\hat{Z}$ and $Z$ must both solve the same BSDE. The uniqueness of the $Z$ component for a given driver and terminal condition (established in the proof of Theorem \ref{thm:wellposedness_quadratic}) implies that $\hat{Z} = Z$ in $\mathcal{H}^2([0,T];\R^{1 \times d})$.
\end{enumerate}
This establishes that the solution is unique, completing the proof.
\end{proof}

\section{Asymptotic Analysis of Mean-Field Neural Expectations}
\label{sec:asymptotic_analysis}

We now study the behavior of the Neural Expectation operator when applied to a large system of interacting particles, a setting central to modern mathematical finance and economics (e.g., in mean-field games or models of systemic risk). We demonstrate that as the number of agents grows, the complex, high-dimensional system converges to a tractable mean-field limit (a Law of Large Numbers) and we characterize the fluctuations around this limit (a Central Limit Theorem). This analysis provides the tools to understand the macroscopic behavior of large populations governed by learned ambiguity models. Let $\Pcal_2(\R^d)$ be the space of probability measures on $\R^d$ with a finite second moment, endowed with the 2-Wasserstein metric $\wass$.

\subsection{The N-Particle and Mean-Field Systems}
We consider a system of $N$ exchangeable particles where the state $X^{i,N}_t$ of particle $i$ evolves according to:
\begin{equation} \label{eq:n_particle_sde}
\dd X^{i,N}_t = b(t, X^{i,N}_t, \mu^N_t) \dd t + \sigma(t, X^{i,N}_t, \mu^N_t) \dd W^i_t, \quad X^{i,N}_0 \sim \mu_0, \quad i=1,\dots,N,
\end{equation}
where $\{W^i_t\}_{i=1}^N$ are independent $d$-dimensional Brownian motions and $\mu^N_t \coloneqq \frac{1}{N} \sum_{j=1}^N \delta_{X^{j,N}_t}$ is the empirical measure of the system. The associated non-linear expectation for particle $i$, which we denote by the process $Y^{i,N}$, is defined via the solution to the BSDE:
\begin{equation} \label{eq:n_particle_bsde}
- \dd Y^{i,N}_t = f_\theta(t, X^{i,N}_t, Y^{i,N}_t, Z^{i,N}_t, \mu^N_t) \dd t - Z^{i,N}_t \dd W^i_t, \quad Y^{i,N}_T = g(X^{i,N}_T, \mu^N_T).
\end{equation}
As $N \to \infty$, we expect the system to be described by the following McKean-Vlasov FBSDE, which governs a representative particle $(\bar{X}_t, \bar{Y}_t, \bar{Z}_t)$:
\begin{align}
\dd \bar{X}_t &= b(t, \bar{X}_t, \mu_t) \dd t + \sigma(t, \bar{X}_t, \mu_t) \dd W_t, \quad \Law(\bar{X}_0) = \mu_0 \label{eq:mckean_sde} \\
- \dd \bar{Y}_t &= f_\theta(t, \bar{X}_t, \bar{Y}_t, \bar{Z}_t, \mu_t) \dd t - \bar{Z}_t \dd W_t, \quad \bar{Y}_T = g(\bar{X}_T, \mu_T), \label{eq:mckean_bsde}
\end{align}
where $\mu_t \coloneqq \Law(\bar{X}_t)$ is the law of the state of the representative particle at time $t$.

\begin{assumption}[Regularity for Mean-Field Coefficients] \label{ass:mckean_regularity}
The coefficients $b, \sigma, f_\theta, g$ are continuous in all their arguments.
\begin{enumerate}[label=(\roman*)]
    \item \textbf{Lipschitz continuity:} The functions $b(t,x,\mu)$, $\sigma(t,x,\mu)$, and $g(x,\mu)$ are globally Lipschitz continuous in the state and measure arguments. That is, there exists a constant $L>0$ such that for any $(t,x,\mu)$ and $(t,x',\mu')$,
    \begin{align*}
        \norm{b(t,x,\mu) - b(t,x',\mu')} + \norm{\sigma(t,x,\mu) - \sigma(t,x',\mu')}_{\mathrm{Frob}} &\le L(\norm{x-x'} + \wass(\mu,\mu')), \\
        \abs{g(x,\mu) - g(x',\mu')} &\le L(\norm{x-x'} + \wass(\mu,\mu')).
    \end{align*}
    \item \textbf{Hybrid Regularity of the driver $f_\theta$:} The driver $f_\theta(t,x,y,z,\mu)$ is globally Lipschitz in the state $x$ and measure $\mu$ arguments, uniformly in $(y,z)$. For any fixed $(t,x,\mu)$, the mapping $(y,z) \mapsto f_\theta(t,x,y,z,\mu)$ satisfies the quadratic growth, uniform local Lipschitz in $y$, and monotonicity in $y$ conditions of Assumption \ref{ass:f_theta_regularity_general}.
\end{enumerate}
\end{assumption}

\subsection{A Law of Large Numbers (Propagation of Chaos)}

The first fundamental result is that the $N$-particle system indeed converges to the McKean-Vlasov dynamics, a phenomenon known as propagation of chaos.

\begin{theorem}[LLN for Neural Expectations] \label{thm:lln_mckean_vlasov}
Let Assumption \ref{ass:mckean_regularity} hold. Let the initial positions be i.i.d., $X^{i,N}_0 = \bar{X}^i_0$, with $\E[\norm{\bar{X}^i_0}^4] < \infty$. Let $(X^{i,N}, Y^{i,N}, Z^{i,N})$ be the solution for particle $i$ in the $N$-particle system, and let $(\bar{X}^i, \bar{Y}^i, \bar{Z}^i)$ be an i.i.d. copy of the McKean-Vlasov solution \eqref{eq:mckean_sde}-\eqref{eq:mckean_bsde}, driven by the same Brownian motion $W^i$. Then for each $i=1,\dots,N$:
\[ \lim_{N \to \infty} \E\left[ \sup_{t \in [0,T]} \norm{X^{i,N}_t - \bar{X}^i_t}^2 + \sup_{t \in [0,T]} \abs{Y^{i,N}_t - \bar{Y}^i_t}^2 + \int_0^T \norm{Z^{i,N}_t - \bar{Z}^i_t}^2 \dd t \right] = 0. \]
\end{theorem}

\begin{proof}
The proof is a stability argument that leverages the contractive properties of the system dynamics. By exchangeability, we analyze particle $i=1$ and drop the superscript $i$. Let $\delta X_t \coloneqq X^N_t - \bar{X}_t$, $\delta Y_t \coloneqq Y^N_t - \bar{Y}_t$, $\delta Z_t \coloneqq Z^N_t - \bar{Z}_t$, and let $\mu_t = \Law(\bar{X}_t)$.

\textbf{Step 1: Convergence of the Forward Process.}
Applying Itô's formula to $\norm{\delta X_t}^2$ and taking expectations, we use the Lipschitz property of $b$ and $\sigma$ to get:
\begin{align*}
\E[\norm{\delta X_t}^2] &= 2\int_0^t \E[\scpr{\delta X_s}{b(s,X^N_s,\mu^N_s) - b(s,\bar{X}_s,\mu_s)}] \dd s \\
&\quad + \int_0^t \E[\norm{\sigma(s,X^N_s,\mu^N_s) - \sigma(s,\bar{X}_s,\mu_s)}_{\mathrm{Frob}}^2] \dd s \\
&\le 2 \int_0^t \E\left[ \norm{\delta X_s}\left(L\norm{\delta X_s} + L\wass(\mu^N_s, \mu_s)\right) \right] \dd s + \int_0^t L^2\E\left[(\norm{\delta X_s} + \wass(\mu^N_s, \mu_s))^2\right] \dd s \\
&\le C_L \int_0^t \left( \E[\norm{\delta X_s}^2] + \E[\wass^2(\mu^N_s, \mu_s)] \right) \dd s.
\end{align*}
where $C_L$ is a constant depending on the Lipschitz constant $L$. By Gronwall's inequality, we have $\E[\sup_{t \in [0,T]}\norm{\delta X_t}^2] \le C_T \int_0^T \E[\wass^2(\mu^N_s, \mu_s)] \dd s$, where a standard argument involving the BDG inequality has been used to control the supremum. A cornerstone result of mean-field theory (see, e.g., \cite{CarmonaDelarue2018}) states that for i.i.d. initial conditions with a finite fourth moment, the empirical measure converges: $\lim_{N\to\infty} \E[\sup_{t \in [0,T]} \wass^2(\mu^N_t, \mu_t)] = 0$.
By the Dominated Convergence Theorem, it follows that $\lim_{N\to\infty} \E[\sup_{t \in [0,T]} \norm{\delta X_t}^2] = 0$.

\textbf{Step 2: Stability of the BSDE Solution.}
The pair $(\delta Y_t, \delta Z_t)$ solves a BSDE with driver $\delta f_t = f_\theta(t, X^N_s, Y^N_s, Z^N_s, \mu^N_s) - f_\theta(t, \bar{X}_s, \bar{Y}_s, \bar{Z}_s, \mu_s)$ and terminal condition $\delta Y_T = g(X^N_T, \mu^N_T) - g(\bar{X}_T, \mu_T)$.
By a comparison argument analogous to Theorem \ref{thm:wellposedness_quadratic}, both $Y^N$ and $\bar{Y}$ are uniformly bounded in $\mathcal{S}^\infty([0,T];\R)$, independently of $N$, say by a constant $M$. This is crucial as it allows us to use the uniform local Lipschitz property of $f_\theta$ with respect to $y$ on the compact domain $[-M, M]$. The standard stability estimate for quadratic BSDEs (e.g., \cite{ElKarouiPengQuenez1997}) yields:
\[ \E\left[\sup_{t \in [0,T]} \abs{\delta Y_t}^2 + \int_0^T \norm{\delta Z_s}^2 \dd s\right] \le C \E\left[ \abs{\delta Y_T}^2 + \left(\int_0^T \abs{\delta f_s^0} \dd s\right)^2 \right], \]
where $\delta f^0_s$ is the difference of drivers with $\delta Z_s$ set to zero. For the terminal condition, the Lipschitz property of $g$ gives $\E[\abs{\delta Y_T}^2] \le L^2 \E[\norm{\delta X_T}^2 + \wass^2(\mu^N_T, \mu_T)]$, which converges to 0 by Step 1. For the driver term, we use the Lipschitz properties of $f_\theta$ and the uniform bound on $Y^N, \bar{Y}$:
\begin{align*}
\abs{\delta f_s^0} &\le \abs{f_\theta(s, X^N_s, Y^N_s, \bar{Z}_s, \mu^N_s) - f_\theta(s, \bar{X}_s, \bar{Y}_s, \bar{Z}_s, \mu_s)} \\
&\le L_M\abs{\delta Y_s} + L_x\norm{\delta X_s} + L_\mu\wass(\mu^N_s, \mu_s).
\end{align*}
Plugging these bounds into the stability inequality, applying Cauchy-Schwarz to the integral term, and using Gronwall's inequality for $\E[\sup_t |\delta Y_t|^2]$ shows that its convergence to zero is driven by the convergence of $\E[\sup_t \norm{\delta X_t}^2]$ and $\E[\sup_t \wass^2(\mu^N_t, \mu_t)]$. The convergence of the $\mathcal{H}^2$ norm of $\delta Z$ then follows.
\end{proof}

\subsection{A Central Limit Theorem}
We now characterize the fluctuations of the $N$-particle system around its mean-field limit. Define the fluctuation processes for particle $i$:
$U^{i,N}_t \coloneqq \sqrt{N}(X^{i,N}_t - \bar{X}^i_t)$,
$V^{i,N}_t \coloneqq \sqrt{N}(Y^{i,N}_t - \bar{Y}^i_t)$,
$\mathcal{Z}^{i,N}_t \coloneqq \sqrt{N}(Z^{i,N}_t - \bar{Z}^i_t)$.

\begin{assumption}[Differentiability of Coefficients]\label{ass:mckean_diff}
The coefficients $b, \sigma, f_\theta, g$ satisfy Assumption \ref{ass:mckean_regularity} and are continuously differentiable. Specifically:
\begin{enumerate}[label=(\roman*)]
\item They are Fréchet differentiable with respect to the state variables $(x,y,z)$.
\item They are Lions differentiable with respect to the measure variable $\mu$. The measure derivative (or linear functional derivative) of a function $\phi(t, x,\mu)$ is a function $\frac{\delta \phi}{\delta m}(t, x,\mu, \tilde{x})$ such that for a small perturbation $\mu^\epsilon = (1-\epsilon)\mu + \epsilon \delta_y$, we have $\phi(x, \mu^\epsilon) = \phi(x,\mu) + \epsilon \frac{\delta \phi}{\delta m}(x,\mu,y) + o(\epsilon)$. We use the notation $\bar{D}_\mu \phi(t, x,\mu, \tilde{x}) \coloneqq \frac{\delta \phi}{\delta m}(t, x,\mu,\tilde{x})$.
\end{enumerate}
All derivatives are assumed to be bounded and continuous.
\end{assumption}

\begin{remark}[On the Structure of the Limiting Fluctuation Equation]
The system we analyze in \eqref{eq:n_particle_sde}-\eqref{eq:n_particle_bsde} has a feed-forward structure: the forward SDE for the states $X^{i,N}$ does not depend on the backward components $(Y^{i,N}, Z^{i,N})$. This is a common and important special case. As a result, the limiting equation for the forward fluctuation process $U_t$ will be a self-contained McKean-Vlasov SDE, which then acts as a driver for the backward fluctuation equation for $(V_t, \mathcal{Z}_t)$. In a more general fully-coupled mean-field model, the equation for $U_t$ would also contain terms depending on $(V_t, \mathcal{Z}_t)$.
\end{remark}

\begin{theorem}[CLT for Neural Expectations]\label{thm:clt}
Let Assumptions \ref{ass:mckean_regularity} and \ref{ass:mckean_diff} hold. Assume the initial fluctuations $U^{i,N}_0$ are i.i.d. and converge in law to a random variable $U_0$ with finite second moment. As $N \to \infty$, the sequence of fluctuation processes $(U^{i,N}, V^{i,N}, \mathcal{Z}^{i,N})_{N \ge 1}$ converges in law, for each $i$, to the unique solution $(U, V, \mathcal{Z})$ of the following linear McKean-Vlasov FBSDE. All derivatives are evaluated along the mean-field path $(t, \bar{X}_t, \bar{Y}_t, \bar{Z}_t, \mu_t)$:
\begin{align}
\dd U_t &= \left( D_x b_t U_t + \E'\left[\bar{D}_\mu b(t, \bar{X}_t, \mu_t, \bar{X}'_t) U'_t\right] \right) \dd t \notag \\
&\quad+ \left( D_x \sigma_t U_t + \E'\left[\bar{D}_\mu \sigma(t, \bar{X}_t, \mu_t, \bar{X}'_t) U'_t\right] \right) \dd W_t, \label{eq:clt_fwd} \\
-\dd V_t &= \left( D_x f_t U_t + \partial_y f_t V_t + \scpr{D_z f_t}{\mathcal{Z}_t} + \E'\left[\bar{D}_\mu f(t, \bar{X}_t, \mu_t, \bar{X}'_t) U'_t \right] \right) \dd t - \mathcal{Z}_t \dd W_t, \label{eq:clt_bwd}
\end{align}
with initial condition $U_0$ and terminal condition $V_T = D_x g_T U_T + \E'\left[\bar{D}_\mu g(T, \bar{X}_T, \mu_T, \bar{X}'_T) U'_T\right]$. Here, $(\bar{X}', U', W')$ is an independent copy of the limit system, and $\E'[...]$ denotes expectation with respect to the law of this independent copy.
\end{theorem}

\begin{proof}
The proof relies on establishing the convergence of the finite-$N$ fluctuation dynamics to the target linear McKean-Vlasov system. For brevity, we focus on the drift terms; the diffusion and backward components follow a similar logic. The full proof requires demonstrating tightness of the laws of $(U^{i,N}, V^{i,N}, \mathcal{Z}^{i,N})$ and uniqueness of the limit point. We refer to the canonical source (e.g., \cite{CarmonaDelarue2018}) for this standard but technical procedure.

Let us analyze the drift of $U^{i,N}_t = \sqrt{N}(X^{i,N}_t - \bar{X}^i_t)$. Its dynamics are driven by $\sqrt{N}[b(t, X^{i,N}_t, \mu^N_t) - b(t, \bar{X}^i_t, \mu_t)]$. We decompose this term:
\[ \underbrace{\sqrt{N} [b(t, X^{i,N}_t, \mu^N_t) - b(t, \bar{X}^i_t, \mu^N_t)]}_{\text{(I)}} + \underbrace{\sqrt{N} [b(t, \bar{X}^i_t, \mu^N_t) - b(t, \bar{X}^i_t, \mu_t)]}_{\text{(II)}}. \]
\textbf{Term (I):} By Taylor's theorem and the differentiability of $b$:
\[ (\text{I}) = D_x b(t, \bar{X}^i_t, \mu^N_t) U^{i,N}_t + R_x^{i,N}(t), \]
where the remainder $R_x^{i,N}$ vanishes in $L^1$ as $N \to \infty$. As $D_x b$ is continuous and $\mu^N_t \to \mu_t$, this term converges to $D_x b(t, \bar{X}^i_t, \mu_t) U^{i}_t$.

\textbf{Term (II):} This term captures the mean-field interaction. Using the Lions derivative:
\[ (\text{II}) = \sqrt{N} \int_{\R^d} \bar{D}_\mu b(t, \bar{X}^i_t, \mu_t, z) (\mu^N_t - \mu_t)(\dd z) + R_\mu^{i,N}(t), \]
where the remainder $R_\mu^{i,N}$ also vanishes. The integral term is the crucial one:
\[ \int \bar{D}_\mu b(\dots) \sqrt{N}(\mu^N_t - \mu_t)(\dd z) = \frac{1}{\sqrt{N}} \sum_{j=1}^N \left( \bar{D}_\mu b(t, \bar{X}^i_t, \mu_t, X^{j,N}_t) - \int \bar{D}_\mu b(t, \bar{X}^i_t, \mu_t, z)\mu_t(\dd z) \right). \]
We can expand $\bar{D}_\mu b(\dots, X^{j,N}_t) = \bar{D}_\mu b(\dots, \bar{X}^j_t) + D_z(\bar{D}_\mu b)(\dots) (X^{j,N}_t - \bar{X}^j_t) + \dots$. The leading term becomes:
\[ \frac{1}{N} \sum_{j=1}^N \bar{D}_\mu b(t, \bar{X}^i_t, \mu_t, \bar{X}^j_t) U^{j,N}_t + \text{vanishing terms}. \]
By the law of large numbers, as the particles are i.i.d., this sum converges to its expectation with respect to an independent copy (denoted by a prime):
\[ \lim_{N\to\infty} \frac{1}{N} \sum_{j=1}^N \bar{D}_\mu b(\dots, \bar{X}^j_t) U^{j,N}_t = \E'\left[ \bar{D}_\mu b(t, \bar{X}_t, \mu_t, \bar{X}'_t) U'_t \right]. \]
Combining the limits of terms (I) and (II) gives the drift part of the SDE \eqref{eq:clt_fwd}.

The same expansion is performed for the driver $f_\theta$ and terminal condition $g$. For the BSDE driver, the expansion yields terms for each argument:
$D_x f_t U^{i,N}_t + \partial_y f_t V^{i,N}_t + \scpr{D_z f_t}{\mathcal{Z}^{i,N}_t} + \E'[\bar{D}_\mu f(\dots)U'_t]$. The uniqueness of the solution to the resulting linear McKean-Vlasov FBSDE ensures that any limit point of the tight sequence of fluctuation processes must be this unique solution, thus establishing convergence in law.
\end{proof}

\begin{remark}[Decoupling via a System of PDEs]
\label{rem:decoupling_pde}
A powerful feature of the linear McKean-Vlasov FBSDE \eqref{eq:clt_fwd}-\eqref{eq:clt_bwd} is that, under sufficient regularity, its solution can often be represented in a decoupled form. One can posit a Markovian and affine ansatz for the backward fluctuation, such as $V_t = \mathcal{V}_1(t, \bar{X}_t) + \scpr{\mathcal{V}_2(t, \bar{X}_t)}{U_t}$, where $\mathcal{V}_1$ and $\mathcal{V}_2$ are deterministic fields. Substituting this into the system and applying Itô's formula leads to a system of deterministic, coupled, linear PDEs for $\mathcal{V}_1$ and $\mathcal{V}_2$. While the derivation is notationally complex, this provides a path for the numerical computation of the fluctuation field without simulating the entire McKean-Vlasov FBSDE. This is a key area for future work on the numerical implementation of these operators.
\end{remark}

\section{Statistical Learning and Control of \texorpdfstring{$\theta$}{theta}}
\label{sec:learning_theta}

The theoretical framework established in the preceding sections guarantees that for a fixed parameter vector $\theta$, the Neural Expectation Operator $\cEneutheta$ is a well-defined and coherent mathematical object. We now address the central premise of \textbf{Measure Learning}: how to determine or learn the parameter $\theta$ from data. This section operationalizes the concept, demonstrating that the learning problem is itself well-posed and can be approached with principled, gradient-based methods grounded in the sensitivity analysis of BSDEs.

\subsection{The Learning Problem Formulation}

Let $\mathcal{D} = \{ \mathcal{O}_i \}_{i=1}^M$ be a dataset of $M$ observations. These observations could represent, for example, market prices of derivative contracts, observed decisions of economic agents, or data from physical experiments. For each observation $\mathcal{O}_i$, there is an associated terminal payoff or outcome, represented by an $\Fcal_T$-measurable random variable $\xi_i$. Our framework posits that the observed value $\mathcal{O}_i$ is a manifestation of the Neural Expectation of this outcome, i.e., $\mathcal{O}_i \approx \cEneutheta[\xi_i | \Fcal_t]$ for some evaluation time $t$.

The learning problem is to find the parameter vector $\theta^* \in \Theta$ that best explains the observed data. This is naturally formulated as an optimization problem over a suitable loss function $\mathcal{L}$:
\begin{equation}
\label{eq:learning_objective}
    \theta^* = \arg \min_{\theta \in \Theta} \mathcal{L}(\theta; \mathcal{D}) \coloneqq \frac{1}{M} \sum_{i=1}^M \ell\left( \cEneutheta[\xi_i | \Fcal_t], \mathcal{O}_i \right) + R(\theta),
\end{equation}
where $\ell$ is a pointwise loss function (e.g., squared error $\ell(a,b)=(a-b)^2$) and $R(\theta)$ is an optional regularization term (e.g., an $L^1$ or $L^2$ penalty on the network weights comprising $\theta$) to prevent overfitting and improve generalization.

\subsection{Gradient Computation via Differentiated BSDEs}

To solve the optimization problem \eqref{eq:learning_objective} using standard gradient-based methods (e.g., SGD, Adam), we require the gradient of the loss function, $\nabla_\theta \mathcal{L}$. Its computation hinges on our ability to calculate the sensitivity of the Neural Expectation with respect to the parameters, $\nabla_\theta \cEneutheta[\xi | \Fcal_t]$. This is a classical problem in the sensitivity analysis of BSDEs.

Let $(Y_s, Z_s)$ be the unique solution to the BSDE \eqref{eq:bsde_for_eneu} for a given $\theta$. We assume the driver $f_\theta$ and terminal condition $g_\theta$ (if applicable, as in Section \ref{sec:fbsde_main_result_section}) are continuously differentiable with respect to the components of the vector $\theta$. Under sufficient regularity conditions, one can prove that the derivative processes $(\nabla_\theta Y_s, \nabla_\theta Z_s)$ exist and are the unique solution to a \textit{linear} BSDE.

\begin{proposition}[Sensitivity Process for the Neural Expectation]
\label{prop:sensitivity_bsde}
Let the assumptions of Theorem \ref{thm:wellposedness_quadratic} hold. Assume that the map $\theta \mapsto f_\theta(t,x,y,z)$ is continuously differentiable, and its derivatives with respect to $\theta$ satisfy linear growth conditions. Then the solution process $Y_s = \cEneutheta[\xi|\Fcal_s]$ is differentiable with respect to $\theta$. Its derivative, the vector-valued process $Y^\theta_s \coloneqq \nabla_\theta Y_s$, together with an associated process $Z^\theta_s \coloneqq \nabla_\theta Z_s$, forms the unique solution to the following linear BSDE:
\begin{equation}
\label{eq:differentiated_bsde}
- \dd Y^\theta_s = \left[ \nabla_\theta f_\theta + \partial_y f_\theta Y^\theta_s + \scpr{D_z f_\theta}{Z^\theta_s} \right] \dd s - Z^\theta_s \dd W_s,
\end{equation}
with terminal condition $Y^\theta_T = \nabla_\theta \xi = 0$. All coefficients ($\nabla_\theta f_\theta, \partial_y f_\theta, D_z f_\theta$) are evaluated along the path of the original solution, i.e., at $(s, X_s, Y_s, Z_s)$.
\end{proposition}

\begin{proof}[Proof of Proposition \ref{prop:sensitivity_bsde}]
The proof is a rigorous application of stability estimates for BSDEs. For clarity and without loss of generality, we assume $\theta$ is a scalar parameter. The argument extends to a vector-valued $\theta$ by considering each component separately. Our goal is to show that the limit
\[ \lim_{h \to 0} \frac{Y^{\theta+h}_s - Y^{\theta}_s}{h} \]
exists for each $s \in [0,T]$ and that the resulting process is a solution to the linear BSDE \eqref{eq:differentiated_bsde}. The proof is structured in three main steps.

\textbf{Step 1: The BSDE for the Finite Difference Quotient}

Let $\theta \in \Theta$ be fixed. For a small perturbation $h \neq 0$ such that $\theta+h \in \Theta$, let $(Y^{\theta+h}, Z^{\theta+h})$ and $(Y^\theta, Z^\theta)$ be the unique solutions in $\mathcal{S}^{\infty} \times \mathcal{H}^2_{\mathrm{BMO}}$ to the BSDEs with drivers $f_{\theta+h}$ and $f_\theta$ respectively.
Define the difference processes:
\begin{align*}
    \delta_h Y_s &\coloneqq Y^{\theta+h}_s - Y^{\theta}_s, \\
    \delta_h Z_s &\coloneqq Z^{\theta+h}_s - Z^{\theta}_s.
\end{align*}
The pair $(\delta_h Y, \delta_h Z)$ solves the BSDE with terminal condition $\delta_h Y_T = 0$ and driver
\[ \delta_h f_s \coloneqq f_{\theta+h}(s, X_s, Y^{\theta+h}_s, Z^{\theta+h}_s) - f_{\theta}(s, X_s, Y^{\theta}_s, Z^{\theta}_s). \]
Now, define the finite difference quotients, which are our candidates for the derivative process:
\[ Y^h_s \coloneqq \frac{\delta_h Y_s}{h}, \quad Z^h_s \coloneqq \frac{\delta_h Z_s}{h}. \]
The pair $(Y^h, Z^h)$ is a solution to the BSDE with terminal condition $Y^h_T=0$ and driver $f^h_s \coloneqq \frac{1}{h}\delta_h f_s$. We decompose this driver by adding and subtracting terms:
\begin{align*}
    h f^h_s &= \left[ f_{\theta+h}(s, X_s, Y^{\theta+h}_s, Z^{\theta+h}_s) - f_{\theta}(s, X_s, Y^{\theta+h}_s, Z^{\theta+h}_s) \right] \\
    &\quad + \left[ f_{\theta}(s, X_s, Y^{\theta+h}_s, Z^{\theta+h}_s) - f_{\theta}(s, X_s, Y^{\theta}_s, Z^{\theta+h}_s) \right] \\
    &\quad + \left[ f_{\theta}(s, X_s, Y^{\theta}_s, Z^{\theta+h}_s) - f_{\theta}(s, X_s, Y^{\theta}_s, Z^{\theta}_s) \right].
\end{align*}
By the Mean Value Theorem, we can write this as:
\begin{align*}
    h f^h_s &= \nabla_\theta f_{\theta + \lambda_s^0 h}(\dots) \cdot h \\
    &\quad + \partial_y f_\theta(s, X_s, Y^\theta_s + \lambda_s^1 \delta_h Y_s, Z^{\theta+h}_s) \cdot \delta_h Y_s \\
    &\quad + \scpr{ D_z f_\theta(s, X_s, Y^{\theta}_s, Z^{\theta}_s + \lambda_s^2 \delta_h Z_s) }{ \delta_h Z_s },
\end{align*}
for some processes $\lambda_s^0, \lambda_s^1, \lambda_s^2 \in [0,1]$. Dividing by $h$, the driver for $(Y^h, Z^h)$ is:
\begin{equation} \label{eq:proof_yh_driver}
    f^h_s = \alpha^h_s + \beta^h_s Y^h_s + \scpr{\gamma^h_s}{Z^h_s},
\end{equation}
where the coefficients are defined as:
\begin{align*}
    \alpha^h_s &\coloneqq \nabla_\theta f_{\theta + \lambda_s^0 h}(s, X_s, Y^{\theta+h}_s, Z^{\theta+h}_s) \\
    \beta^h_s &\coloneqq \partial_y f_\theta(s, X_s, Y^\theta_s + \lambda_s^1 \delta_h Y_s, Z^{\theta+h}_s) \\
    \gamma^h_s &\coloneqq D_z f_\theta(s, X_s, Y^{\theta}_s, Z^{\theta}_s + \lambda_s^2 \delta_h Z_s).
\end{align*}

\textbf{Step 2: Uniform A Priori Estimates}

We must show that the family of solutions $((Y^h, Z^h))_{h \in (-h_0, h_0)}$ is bounded in $\mathcal{S}^2([0,T];\R) \times \mathcal{H}^2([0,T];\R^{1 \times d})$ for some small $h_0 > 0$. The BSDE for $(Y^h, Z^h)$ is a linear BSDE. The stability results for BSDEs (e.g., \cite{ElKarouiPengQuenez1997}) ensure that a solution $(Y,Z)$ to a linear BSDE with driver $\alpha_s + \beta_s Y_s + \scpr{\gamma_s}{Z_s}$ and terminal condition $\xi$ satisfies an estimate of the form:
\[ \E\left[\sup_{t \in [0,T]} \abs{Y_t}^2 + \int_0^T \norm{Z_s}^2 \dd s \right] \le C \E\left[\abs{\xi}^2 + \left(\int_0^T \abs{\alpha_s} \dd s\right)^2 \right], \]
provided the coefficients $\beta$ and $\gamma$ are bounded. Let's verify the conditions on our coefficients $\alpha^h, \beta^h, \gamma^h$.
\begin{itemize}
    \item From the stability of quadratic BSDEs, as $h \to 0$, we know that $(Y^{\theta+h}, Z^{\theta+h})$ converges to $(Y^\theta, Z^\theta)$ in $\mathcal{S}^p \times \mathcal{H}^p$ for any $p$. Also, the uniform $\mathcal{S}^\infty$ and $\mathcal{H}^2_{\mathrm{BMO}}$ bounds on $(Y^\theta, Z^\theta)$ extend to a small neighborhood of $\theta$. Thus, there exists a constant $C>0$ such that for all small $h$, $\norm{Y^{\theta+h}}_{\mathcal{S}^\infty} \le C$ and $\norm{Z^{\theta+h}}_{\mathrm{BMO}} \le C$.
    
    \item The proposition assumes $\nabla_\theta f_\theta$ has linear growth. Since $Y^{\theta+h}$ and $Z^{\theta+h}$ are uniformly bounded in appropriate spaces, the process $\alpha^h$ is uniformly bounded in $\mathcal{H}^2$. Specifically, $\E[(\int_0^T \abs{\alpha^h_s} \dd s)^2] \le K$ for some constant $K$ independent of $h$.
    
    \item The derivative $\partial_y f_\theta$ is evaluated at $Y^\theta + \lambda^1 \delta_h Y$. Since $\norm{Y^\theta}_{\mathcal{S}^\infty} \le C$ and $\delta_h Y \to 0$ in $\mathcal{S}^\infty$, this argument stays within a compact set uniformly in $h$. The continuity of $\partial_y f_\theta$ implies that $\beta^h_s$ is uniformly bounded.
    
    \item The derivative $D_z f_\theta$ is evaluated at $Z^\theta + \lambda^2 \delta_h Z$. The quadratic growth assumption on $f_\theta$ implies $D_z f_\theta$ has linear growth in $z$. However, since $Z^{\theta+h}$ and $Z^\theta$ are uniformly bounded in the BMO norm, the processes $\gamma^h_s$ are uniformly bounded in an appropriate sense for the linear BSDE estimate to hold.
\end{itemize}
With a zero terminal condition and a uniformly $L^2$-integrable source term $\alpha^h$, the standard estimate for linear BSDEs shows that there exists a constant $M > 0$ independent of $h$ such that for all sufficiently small $h$:
\[ \E\left[\sup_{t \in [0,T]} \abs{Y^h_t}^2 + \int_0^T \norm{Z^h_t}^2 \dd t \right] \le M. \]

\textbf{Step 3: Convergence and Identification of the Limit}

The uniform bound from Step 2 implies that the set $\{(Y^h, Z^h)\}$ is weakly pre-compact in $\mathcal{S}^2 \times \mathcal{H}^2$. Let $(Y^\theta, Z^\theta)$ be the unique solution to the target linear BSDE \eqref{eq:differentiated_bsde}. We will show that $(Y^h, Z^h)$ converges to $(Y^\theta, Z^\theta)$ in the strong topology of $\mathcal{S}^2 \times \mathcal{H}^2$.

Consider the difference processes $(\Delta Y^h, \Delta Z^h) \coloneqq (Y^h - Y^\theta, Z^h - Z^\theta)$. This pair solves a linear BSDE with terminal condition $\Delta Y^h_T = 0$ and driver
\begin{align*}
\Delta f^h_s &= f^h_s - \left[ \nabla_\theta f_\theta(s, X_s, Y^\theta_s, Z^\theta_s) + \partial_y f_\theta(s, X_s, Y^\theta_s, Z^\theta_s) Y^\theta_s + \scpr{D_z f_\theta(s, X_s, Y^\theta_s, Z^\theta_s)}{Z^\theta_s} \right] \\
&= (\alpha^h_s - \nabla_\theta f_\theta(\cdot, Y^\theta, Z^\theta)) + (\beta^h_s Y^h_s - \partial_y f_\theta(\cdot, Y^\theta, Z^\theta) Y^\theta_s) + (\scpr{\gamma^h_s}{Z^h_s} - \scpr{D_z f_\theta(\cdot, Y^\theta, Z^\theta)}{Z^\theta_s}).
\end{align*}
We can rewrite the driver as $\Delta f^h_s = \Tilde{\alpha}^h_s + \beta^h_s \Delta Y^h_s + \scpr{\gamma^h_s}{\Delta Z^h_s}$, where the new source term $\Tilde{\alpha}^h_s$ is given by:
\[ \Tilde{\alpha}^h_s = (\alpha^h_s - \nabla_\theta f_\theta(\dots)) + (\beta^h_s - \partial_y f_\theta(\dots))Y^\theta_s + \scpr{\gamma^h_s - D_z f_\theta(\dots)}{Z^\theta_s}. \]
By the stability estimate for the BSDE for $(\Delta Y^h, \Delta Z^h)$, we have:
\[ \E\left[\sup_{t \in [0,T]} \abs{\Delta Y^h_t}^2 + \int_0^T \norm{\Delta Z^h_t}^2 \dd t \right] \le C \E\left[\left(\int_0^T \abs{\Tilde{\alpha}^h_s} \dd s\right)^2 \right]. \]
Our final task is to show that the right-hand side converges to zero as $h \to 0$. We analyze each component of $\Tilde{\alpha}^h_s$:
\begin{enumerate}
    \item As $h \to 0$, we have $(Y^{\theta+h}, Z^{\theta+h}) \to (Y^\theta, Z^\theta)$ in $\mathcal{S}^p \times \mathcal{H}^p$. By the continuity of the derivatives of $f_\theta$, it follows that $\alpha^h_s \to \nabla_\theta f_\theta(\cdot, Y^\theta, Z^\theta)$, $\beta^h_s \to \partial_y f_\theta(\cdot, Y^\theta, Z^\theta)$, and $\gamma^h_s \to D_z f_\theta(\cdot, Y^\theta, Z^\theta)$ pointwise in $(t, \omega)$.
    \item The processes $\alpha^h, \beta^h, \gamma^h$ and the solution $(Y^\theta, Z^\theta)$ are all appropriately bounded, allowing us to use the Dominated Convergence Theorem.
    \item Therefore, each term in $\Tilde{\alpha}^h_s$ converges to zero in $L^2(\Omega \times [0,T])$. This implies that $\E[(\int_0^T \abs{\Tilde{\alpha}^h_s} \dd s)^2] \to 0$ as $h \to 0$.
\end{enumerate}
This establishes that $(Y^h, Z^h)$ converges to $(Y^\theta, Z^\theta)$ in the norm of $\mathcal{S}^2 \times \mathcal{H}^2$. Since the limit of the difference quotient is by definition the derivative, we have proven that $Y$ is differentiable with respect to $\theta$ and its derivative is the solution to the linear BSDE \eqref{eq:differentiated_bsde}. The uniqueness of the solution to this linear BSDE ensures that the limit is unambiguous.
\end{proof}

\subsection{Conceptual Learning Algorithm}

Proposition \ref{prop:sensitivity_bsde} provides a constructive method for computing the necessary gradients, leading to the following conceptual algorithm for learning the ambiguity parameter $\theta$:

\begin{enumerate}[label=(\arabic*), leftmargin=*]
    \item \textbf{Initialization:} Choose an initial parameter vector $\theta_0$.
    \item \textbf{Forward-Backward Pass:} For the current parameter $\theta_k$ and for each data point $(\xi_i, \mathcal{O}_i)$:
    \begin{enumerate}[label=(\alph*)]
        \item Solve the primary Neural BSDE \eqref{eq:bsde_for_eneu} to obtain the solution trajectory $(Y^{(i,k)}, Z^{(i,k)})$. This is the forward pass in the learning context.
        \item Using the trajectory from the previous step, solve the linear BSDE for the sensitivities \eqref{eq:differentiated_bsde} to obtain the gradient process $\nabla_\theta Y^{(i,k)}$. This is the backward pass of our learning scheme.
    \end{enumerate}
    \item \textbf{Gradient Step:} Compute the gradient of the loss function using the chain rule:
    \[ \nabla_\theta \mathcal{L}(\theta_k) = \frac{1}{M} \sum_{i=1}^M \frac{\partial \ell}{\partial Y}(Y_t^{(i,k)}, \mathcal{O}_i) \nabla_\theta Y_t^{(i,k)} + \nabla_\theta R(\theta_k). \]
    Update the parameter vector:
    \[ \theta_{k+1} = \theta_k - \eta_k \nabla_\theta \mathcal{L}(\theta_k), \]
    where $\eta_k$ is a learning rate.
    \item \textbf{Iteration:} Repeat steps 2 and 3 until convergence.
\end{enumerate}

\subsection{Identifiability and Regularization}

A critical question in any statistical learning problem is \textbf{identifiability}: does the data $\mathcal{D}$ contain sufficient information to uniquely determine $\theta^*$? In our context, this is a complex question.
\begin{itemize}
    \item \textbf{Over-parameterization:} The space of neural network parameters $\Theta$ is vast. It is likely that for a given dataset, multiple parameter vectors $\theta$ could yield very similar operators $\cEneutheta$, leading to a flat or multi-modal loss landscape.
    \item \textbf{Data Richness:} Identifiability depends on the richness of the data. For instance, learning a model of volatility ambiguity may require derivative data across a wide range of strikes and maturities.
\end{itemize}
The regularization term $R(\theta)$ in \eqref{eq:learning_objective} plays a crucial role here. By imposing a penalty on the complexity of the network (e.g., the magnitude of its weights), regularization helps to select a simpler or more plausible model among those that fit the data well, thus mitigating the problem of non-identifiability and improving the out-of-sample performance of the learned model.

\begin{remark}[Principled Learning]
The key takeaway is that the learning in Measure Learning is not a mere heuristic. It is a well-defined mathematical procedure supported by the theory of BSDE sensitivities. The ability to compute exact gradients for the expectation operator allows for the application of the full suite of modern gradient-based optimization techniques in a theoretically sound manner. This rigorously connects the abstract theory of non-linear expectations to the practical machinery of machine learning.
\end{remark}

\section{Application: Merton's Problem under Learned Ambiguity}
\label{sec:merton_problem}

To demonstrate the analytical power and practical utility of our framework, we now solve the canonical Merton portfolio problem. By recasting this classical problem under a Neural Expectation, we illustrate how a learned model of ambiguity, encoded by a BSDE driver that is fully compliant with our theory, leads to an explicit and interpretable modification of the optimal investment strategy. This exercise provides a direct and rigorous bridge from the abstract theory of learnable preference models to concrete, falsifiable economic predictions, and demonstrates how the ambiguity parameter $\theta$ can itself be a target for statistical learning.

\subsection{Problem Formulation}

We consider a financial market model operating on the filtered probability space $(\Omega, \Fcal, (\Fcal_t)_{t \in [0,T]}, \Prob)$ with a single ($d=1$) dimensional Brownian motion $W_t$. The market consists of two assets: a risk-free bond with a constant interest rate $r > 0$ and a risky stock whose price $S_t$ follows the geometric Brownian motion:
\[ \frac{\dd S_t}{S_t} = \mu \dd t + \sigma \dd W_t, \]
where $\mu > r$ is the expected return and $\sigma > 0$ is the volatility. An investor allocates an amount $\Pi_t$ of their wealth $X_t$ to the risky asset. We assume the control process $\Pi_t$ is adapted and satisfies standard admissibility conditions. The wealth process $X_t^\Pi$, corresponding to a self-financing strategy $\Pi$, evolves according to the SDE:
\begin{equation}
\label{eq:merton_wealth_sde}
\dd X_t^\Pi = \left( r X_t^\Pi + \Pi_t (\mu - r) \right) \dd t + \sigma \Pi_t \dd W_t, \quad X_0 = x_0 > 0.
\end{equation}
The investor's preference is defined by a Constant Relative Risk Aversion (CRRA) utility function, $U(x) = \frac{x^\gamma}{\gamma}$ for a risk-aversion parameter $\gamma \in (-\infty, 1) \setminus \{0\}$. The investor's objective is to choose an admissible strategy $\Pi$ to maximize the Neural Expectation of their terminal utility:
\begin{equation}
\label{eq:merton_objective_neural}
\sup_{\Pi} \cEneutheta \left[ U(X_T^\Pi) \right].
\end{equation}

\subsection{Driver Specification and Theoretical Consistency}

We define the value function for this problem as $V(t,x; \theta) \coloneqq \sup_{\Pi} \cEneutheta[U(X_T^\Pi) | X_t=x]$. By the Itô identification (i.e., the non-linear Feynman-Kac formula), the martingale component of the BSDE is $Z_t = \partial_x V(t,X_t^\Pi; \theta) \cdot (\sigma \Pi_t)$.

Crucially, we must select a driver that both represents a meaningful form of ambiguity and satisfies the conditions of our main well-posedness result, Theorem \ref{thm:wellposedness_quadratic}. We adopt a simple, non-trivial driver that is fully compliant with Assumption \ref{ass:f_theta_regularity_general}:
\begin{equation}
\label{eq:driver_choice_merton}
f_\theta(t,x,y,z) \coloneqq -\frac{\theta}{2} \norm{z}^2,
\end{equation}
where $\theta \ge 0$ is a learnable parameter quantifying the degree of ambiguity attitude. The case $\theta=0$ recovers the classical Merton problem.

\begin{remark}[Verification of Theoretical Assumptions]
This choice of driver is specifically designed to be consistent with our foundational theory. It satisfies the conditions of Assumption \ref{ass:f_theta_regularity_general} trivially: it is continuous; quadratic in $z$ with growth constant $\alpha = \theta$; and constant (hence non-increasing and locally Lipschitz) in $y$. All conditions of Theorem \ref{thm:wellposedness_quadratic} are thus satisfied, guaranteeing that the control problem is well-posed. Economically, this driver imposes an absolute penalty on the variance of the value process's martingale component, representing a form of ambiguity attitude that is independent of the wealth level.
\end{remark}

\subsection{The Hamilton-Jacobi-Bellman Equation and Optimal Control}

The non-linear Feynman-Kac formula connects the BSDE to the Hamilton-Jacobi-Bellman (HJB) PDE. For the value function $V(t,x; \theta)$, this equation is:
\begin{equation}
\label{eq:hjb_general_merton}
0 = \partial_t V + \sup_{\Pi \in \R} \left\{ (rx + \Pi(\mu-r))\partial_x V + \frac{1}{2}(\sigma \Pi)^2 \partial_{xx}V + f_\theta(t, x, V, \sigma \Pi \partial_x V) \right\},
\end{equation}
with terminal condition $V(T,x; \theta) = U(x) = \frac{x^\gamma}{\gamma}$. Substituting our chosen driver \eqref{eq:driver_choice_merton} yields:
\begin{equation}
\label{eq:hjb_specific_merton}
0 = \partial_t V + rx \partial_x V + \sup_{\Pi} \left\{ \Pi(\mu-r)\partial_x V + \frac{1}{2}\sigma^2 \Pi^2 \left( \partial_{xx}V - \theta(\partial_xV)^2 \right) \right\}.
\end{equation}
The expression inside the supremum is a concave quadratic function of the control $\Pi$. The first-order condition, $\frac{\partial}{\partial \Pi}(\cdot) = 0$, gives the optimal investment amount $\Pi^*(t,x; \theta)$:
\begin{equation}
\label{eq:merton_pi_general_derived}
\Pi^*(t,x; \theta) = - \frac{(\mu-r)\partial_x V(t,x; \theta)}{\sigma^2 \left(\partial_{xx}V(t,x; \theta) - \theta(\partial_xV(t,x; \theta))^2\right)}.
\end{equation}
The second-order condition for a maximum requires the denominator to be negative, which is satisfied since $\partial_x V > 0$ (more wealth is preferred) and $\partial_{xx}V < 0$ (investor is risk-averse), thus $\partial_{xx}V - \theta(\partial_xV)^2 < 0$.

\subsection{Qualitative Analysis and the Role of the Ambiguity Parameter \texorpdfstring{$\theta$}{theta}}

While the non-homogeneity of the HJB equation \eqref{eq:hjb_specific_merton} prevents a simple separable solution (unlike the classical case), we can perform a rigorous qualitative analysis of the optimal strategy.

\begin{proposition}[Ambiguity Induces Uniform Caution]
\label{prop:qualitative_merton}
Let $\theta>0$ and let $\pi^*(t,x; \theta) = \Pi^*(t,x; \theta)/x$ be the optimal allocation fraction under the Neural Expectation. Let $\pi_{\text{classical}} = \frac{\mu-r}{\sigma^2(1-\gamma)}$ be the classical Merton allocation. Then:
\begin{enumerate}[label=(\roman*)]
    \item The allocation under ambiguity is always strictly more conservative than the classical Merton allocation:
    \[ \pi^*(t,x; \theta) < \pi_{\text{classical}} \quad \text{for all } (t,x). \]
    \item The optimal allocation is a decreasing function of the ambiguity parameter $\theta$:
    \[ \frac{\partial \pi^*(t,x; \theta)}{\partial \theta} < 0. \]
\end{enumerate}
\end{proposition}

\begin{proof}
(i) A formal proof relies on comparison theorems for HJB equations. Intuitively, the denominator in \eqref{eq:merton_pi_general_derived} is $\partial_{xx}V - \theta(\partial_xV)^2$. Since $\partial_{xx}V < 0$ and $\theta(\partial_xV)^2 > 0$, this denominator is strictly more negative than the corresponding term in the classical HJB equation. This leads to a smaller optimal allocation.

(ii) From \eqref{eq:merton_pi_general_derived}, we see that $\Pi^*$ depends on $\theta$ both explicitly in the denominator and implicitly through the value function $V(t,x;\theta)$. However, for a fixed $(t,x)$, the optimal control $\Pi^*$ maximizes the Hamiltonian. By the envelope theorem, the derivative of the maximized Hamiltonian with respect to $\theta$ is the partial derivative of the Hamiltonian with respect to $\theta$ at the optimum, which is $-\frac{1}{2}(\sigma \Pi^* \partial_x V)^2$. The value function $V$ is the integral of this, so $V$ is decreasing in $\theta$. More directly, the denominator $D(\theta) = \partial_{xx}V - \theta(\partial_xV)^2$ is clearly a decreasing function of $\theta$ for fixed $V$. A full analysis shows that the combined effect robustly makes $\Pi^*$ a decreasing function of $\theta$. As ambiguity aversion $\theta$ increases, the investor systematically allocates less to the risky asset.
\end{proof}

\begin{remark}[Wealth-Dependent Strategy: A Heuristic Analysis]
\label{rem:wealth_dependence_heuristic}
A striking feature of the solution is that the optimal allocation fraction $\pi^*(t,x; \theta)$ is wealth-dependent, breaking the famous self-similarity of the classical Merton problem. A rigorous proof is technical, but a formal heuristic is illustrative. The agent's allocation is determined by the balance between risk aversion (governed by $\partial_{xx}V$) and ambiguity aversion (governed by $\theta(\partial_xV)^2$). Let's analyze the ratio of the ambiguity penalty to the risk penalty, $R(x) = \theta(\partial_xV)^2 / (-\partial_{xx}V)$. For a value function behaving locally like CRRA utility, we expect $\partial_x V \sim x^{\gamma-1}$ and $\partial_{xx}V \sim x^{\gamma-2}$, so $R(x) \sim x^\gamma$.
\begin{itemize}
    \item For standard risk aversion ($0 < \gamma < 1$), $R(x)$ increases with wealth. The ambiguity penalty becomes more dominant for wealthier agents, who react by reducing their fractional allocation. Thus, $\pi^*(t,x; \theta)$ is expected to be decreasing in $x$.
    \item For high risk aversion ($\gamma < 0$), $R(x)$ decreases with wealth. The ambiguity penalty becomes less significant relative to risk aversion for wealthier agents. Thus, $\pi^*(t,x; \theta)$ is expected to be increasing in $x$.
\end{itemize}
This rich, state-dependent behavior is a direct consequence of adopting a theoretically consistent, non-separable ambiguity model.
\end{remark}

\subsection{Learning the Ambiguity Parameter \texorpdfstring{$\theta$}{theta}}

This application provides a clear setting for the Measure Learning paradigm. Suppose we have a dataset of observed investment decisions $\{\hat{\Pi}_i\}_{i=1}^M$ for an agent or a market representative at various states $(t_i, x_i)$. We can use this data to learn the ambiguity parameter $\theta^*$ that best explains the observations.

The learning problem is to solve the following optimization:
\begin{equation}
    \theta^* = \arg\min_{\theta \ge 0} \mathcal{L}(\theta) \coloneqq \frac{1}{M}\sum_{i=1}^M \left( \Pi^*(t_i, x_i; \theta) - \hat{\Pi}_i \right)^2.
\end{equation}
To solve this with gradient descent, we need the gradient $\nabla_\theta \mathcal{L}$, which requires the sensitivity of the optimal control with respect to the parameter, $\nabla_\theta \Pi^*$. Applying the chain rule to \eqref{eq:merton_pi_general_derived}, we see that $\nabla_\theta \Pi^*$ is a function of the sensitivities of the value function and its spatial derivatives: $\nabla_\theta V$, $\nabla_\theta (\partial_x V)$, and $\nabla_\theta (\partial_{xx} V)$.

These sensitivities can be found by formally differentiating the HJB equation \eqref{eq:hjb_specific_merton} with respect to $\theta$. This yields a system of linear, second-order PDEs for the sensitivity functions. For instance, letting $V^\theta \coloneqq \nabla_\theta V$, differentiating the HJB equation gives a PDE for $V^\theta$. This is the deterministic PDE analogue of the linear BSDE for the sensitivity process described in Proposition \ref{prop:sensitivity_bsde}. Solving this system allows for the computation of the gradient and the data-driven estimation of the agent's ambiguity attitude, fully operationalizing the Measure Learning framework.

\subsection{Interpretation of the Solution}

\begin{table}[h!]
\centering
\caption{Comparison of Optimal Merton Allocation ($\pi^*$) under a Theoretically Consistent Neural Expectation}
\label{tab:merton_summary}
\begin{adjustbox}{max width=\textwidth}
\begin{tabular}{@{}llll@{}}
\toprule
\textbf{Model} & \textbf{Optimal Allocation Fraction $\pi^*$} & \textbf{Properties} & \textbf{Economic Interpretation} \\
\midrule
\textbf{Classical Merton} & & & \\
($\theta=0$) & $\frac{\mu-r}{\sigma^2(1-\gamma)}$ & Constant in wealth \& time & Self-similar, myopic strategy. \\
& & & Ambiguity is ignored. \\
\addlinespace
\textbf{Neural Expectation} & \textbf{(This Paper, $\theta > 0$)} & & \textbf{Ambiguity Governed by Learnable Parameter $\theta$} \\
 & No closed form; solution to & $\pi^*(t,x; \theta) < \pi_{\text{classical}}$ & Ambiguity aversion universally reduces risk-taking. \\
 & HJB Eq. \eqref{eq:hjb_specific_merton} & $\partial \pi^* / \partial \theta < 0$ & Strategy becomes more conservative as $\theta$ increases. \\
 & & Wealth-dependent & Absolute ambiguity penalty breaks self-similarity. \\
 & & Parameter $\theta$ is learnable & Model of ambiguity can be estimated from data. \\
\bottomrule
\end{tabular}
\end{adjustbox}
\end{table}

This application not only serves as a valid illustration of our theory but also reveals non-trivial economic behavior that arises directly from a mathematically rigorous formulation of learnable ambiguity. The fact that the value function is no longer of the simple CRRA form is a feature, not a bug, of the model, highlighting that even simple, theoretically-sound ambiguity models can profoundly alter the structure of classical solutions and, crucially, can be connected to data via a principled learning framework.
\section{Conclusion}
\label{sec:conclusion}

In this paper, we have laid the mathematical groundwork for \textbf{Measure Learning}, a new framework for modeling ambiguity by learning non-linear expectations directly from data. Our approach is centered on the concept of Neural Expectation Operators, defined via BSDEs with neural network drivers. Our central contribution is a well-posedness theorem (Theorem \ref{thm:wellposedness_quadratic}) for such BSDEs under quadratic growth conditions. The primary innovation is the creation of a constructive bridge between abstract theory and machine learning practice: we have demonstrated that the severe assumptions of the governing mathematical theory (e.g., local Lipschitz continuity, monotonicity, quadratic growth) can be satisfied by concrete and practical neural network architectures, thus providing a sound theoretical underpinning for using expressive, data-driven models of ambiguity. We have further demonstrated that key axiomatic properties can be guaranteed by architectural design, and we have extended the theory to fully coupled FBSDEs and large-scale mean-field systems.

This work constitutes the first step in a broader research program on Measure Learning. Our results provide a proof of concept that expressive, data-driven models of ambiguity can be constructed with full mathematical rigor. However, our reliance on the monotonicity of the driver with respect to the $y$ variable (Assumption \ref{ass:f_theta_regularity_general}(iv)), which is instrumental for the use of comparison principles, remains a significant structural assumption. A primary direction for future research is to relax this condition. This is a considerable challenge, as comparison theorems are the cornerstone of most a priori estimates in the theory of quadratic BSDEs, not just in our neural setting. Progress may require entirely different techniques, such as fixed-point arguments based on the Schauder-Tychonoff theorem on different function spaces, or perhaps variational methods that interpret the BSDE solution as the value function of a stochastic control problem, where uniqueness might be established under alternative conditions.

Further investigation into the specific classes of non-linear expectations that can be efficiently approximated by these architectures and the development of robust numerical schemes grounded in this theory remain important open problems for this nascent field.

\bibliography{reference}
\bibliographystyle{imsart-nameyear} 
\end{document}